\documentclass[11pt]{amsart}
\usepackage{amsopn}
\usepackage{amssymb, amscd}
\usepackage{multirow}

\newcommand{\nc}{\newcommand}

\nc{\fg}{\mathfrak{f} } \nc{\vg}{\mathfrak{v} } \nc{\wg}{\mathfrak{w} }
\nc{\zg}{\mathfrak{z} } \nc{\ngo}{\mathfrak{n} } \nc{\kg}{\mathfrak{k} }
\nc{\mg}{\mathfrak{m} } \nc{\bg}{\mathfrak{b} } \nc{\ggo}{\mathfrak{g} }
\nc{\ggob}{\overline{\mathfrak{g}} } \nc{\sog}{\mathfrak{so} }
\nc{\sug}{\mathfrak{su} } \nc{\spg}{\mathfrak{sp} } \nc{\slg}{\mathfrak{sl} }
\nc{\glg}{\mathfrak{gl} } \nc{\cg}{\mathfrak{c} } \nc{\rg}{\mathfrak{r} }
\nc{\hg}{\mathfrak{h} } \nc{\tg}{\mathfrak{t} } \nc{\ug}{\mathfrak{u} }
\nc{\dg}{\mathfrak{d} } \nc{\ag}{\mathfrak{a} } \nc{\pg}{\mathfrak{p} }
\nc{\sg}{\mathfrak{s} } \nc{\affg}{\mathfrak{aff} } \nc{\qg}{\mathfrak{q} }

\nc{\pca}{\mathcal{P}} \nc{\nca}{\mathcal{N}} \nc{\lca}{\mathcal{L}}
\nc{\oca}{\mathcal{O}} \nc{\mca}{\mathcal{M}} \nc{\tca}{\mathcal{T}}
\nc{\aca}{\mathcal{A}} \nc{\cca}{\mathcal{C}} \nc{\gca}{\mathcal{G}}
\nc{\sca}{\mathcal{S}} \nc{\hca}{\mathcal{H}} \nc{\bca}{\mathcal{B}}
\nc{\dca}{\mathcal{D}} \nc{\val}{\operatorname{val}}

\nc{\vp}{\varphi} \nc{\ddt}{\frac{d}{dt}} \nc{\dds}{\frac{d}{ds}}
\nc{\dpar}{\frac{\partial}{\partial t}} \nc{\im}{\mathtt{i}}

\nc{\SO}{\mathrm{SO}} \nc{\Spe}{\mathrm{Sp}} \nc{\Sl}{\mathrm{SL}}
\nc{\SU}{\mathrm{SU}} \nc{\Or}{\mathrm{O}} \nc{\U}{\mathrm{U}} \nc{\Gl}{\mathrm{GL}}
\nc{\Se}{\mathrm{S}} \nc{\Cl}{\mathrm{Cl}} \nc{\Spein}{\mathrm{Spin}}
\nc{\Pin}{\mathrm{Pin}} \nc{\G}{\mathrm{GL}_n(\RR)} \nc{\g}{\mathfrak{gl}_n(\RR)}

\nc{\RR}{{\Bbb R}} \nc{\HH}{{\Bbb H}} \nc{\CC}{{\Bbb C}} \nc{\ZZ}{{\Bbb Z}}
\nc{\FF}{{\Bbb F}} \nc{\NN}{{\Bbb N}} \nc{\QQ}{{\Bbb Q}} \nc{\PP}{{\Bbb P}} \nc{\OO}{{\Bbb O}}

\nc{\vs}{\vspace{.2cm}} \nc{\vsp}{\vspace{1cm}} \nc{\ip}{\langle\cdot,\cdot\rangle}
\nc{\ipp}{(\cdot,\cdot)} \nc{\la}{\langle} \nc{\ra}{\rangle} \nc{\unm}{\frac{1}{2}}
\nc{\unc}{\frac{1}{4}} \nc{\und}{\frac{1}{16}} \nc{\no}{\vs\noindent}
\nc{\lam}{\Lambda^2(\RR^n)^*\otimes\RR^n} \nc{\tangz}{{\rm T}^{\rm Zar}}
\nc{\nor}{{\sf n}}  \nc{\mum}{/\!\!/} \nc{\kir}{/\!\!/\!\!/}
\nc{\Ri}{\tfrac{4\Ric_{\mu}}{||\mu||^2}} \nc{\ds}{\displaystyle}
\nc{\ben}{\begin{enumerate}} \nc{\een}{\end{enumerate}} \nc{\f}{\frac}
\nc{\lb}{[\cdot,\cdot]} \nc{\isn}{\tfrac{1}{||v||^2}}
\nc{\gkp}{(\ggo=\kg\oplus\pg,\ip)} \nc{\ukh}{(\ug=\kg\oplus\hg,\ip)}
\nc{\tgkp}{(\tilde{\ggo}=\kg\oplus\pg,\ip)}
\nc{\wt}{\widetilde} \nc{\mm}{M}
\nc{\iop}{\mathtt{i}} \nc{\jop}{\mathtt{j}}

\nc{\Hess}{\operatorname{Hess}} \nc{\ad}{\operatorname{ad}}
\nc{\Ad}{\operatorname{Ad}} \nc{\rank}{\operatorname{rank}}
\nc{\Irr}{\operatorname{Irr}} \nc{\End}{\operatorname{End}}
\nc{\Aut}{\operatorname{Aut}} \nc{\Inn}{\operatorname{Inn}}
\nc{\Der}{\operatorname{Der}} \nc{\Ker}{\operatorname{Ker}}
\nc{\Iso}{\operatorname{Iso}} \nc{\Diff}{\operatorname{Diff}}
\nc{\Lie}{\operatorname{L}} \nc{\tr}{\operatorname{tr}} \nc{\dif}{\operatorname{d}}
\nc{\sen}{\operatorname{sen}} \nc{\modu}{\operatorname{mod}}
\nc{\CRic}{\operatorname{PP}} \nc{\Cric}{\operatorname{P}} \nc{\Ricci}{\operatorname{Ric}}
\nc{\sym}{\operatorname{sym}} \nc{\herm}{\operatorname{herm}} \nc{\symac}{\operatorname{sym^{ac}}}
\nc{\symc}{\operatorname{sym^{c}}} \nc{\scalar}{\operatorname{sc}}
\nc{\grad}{\operatorname{grad}} \nc{\ricci}{\operatorname{Rc}}
\nc{\Nor}{\operatorname{Norm}}  \nc{\ricc}{\operatorname{Rc^{c}}}
\nc{\Ricc}{\operatorname{Ric^{c}}} \nc{\ricac}{\operatorname{Rc^{ac}}}
\nc{\Ricac}{\operatorname{Ric^{ac}}} \nc{\Riem}{\operatorname{Rm}}
\nc{\riccig}{\operatorname{ric^{\gamma}}} \nc{\Rin}{\operatorname{M}}
\nc{\Le}{\operatorname{L}} \nc{\tang}{\operatorname{T}}
\nc{\level}{\operatorname{level}} \nc{\rad}{\operatorname{r}}
\nc{\abel}{\operatorname{ab}} \nc{\CH}{\operatorname{CH}}
\nc{\mcc}{\operatorname{mcc}} \nc{\Adj}{\operatorname{Adj}}
\nc{\Order}{\operatorname{O}}  \nc{\inj}{\operatorname{inj}} \nc{\proy}{\operatorname{pr}}
\nc{\vol}{\operatorname{vol}} \nc{\Diag}{\operatorname{Diag}}
\nc{\Spec}{\operatorname{Spec}} \nc{\Ima}{\operatorname{Im}} \nc{\Rea}{\operatorname{Re}}
\nc{\spann}{\operatorname{span}}

\theoremstyle{plain}
\newtheorem{theorem}{Theorem}[section]
\newtheorem{proposition}[theorem]{Proposition}
\newtheorem{corollary}[theorem]{Corollary}
\newtheorem{lemma}[theorem]{Lemma}

\theoremstyle{definition}
\newtheorem{definition}[theorem]{Definition}
\newtheorem{note}[theorem]{Note}

\theoremstyle{remark}
\newtheorem{remark}[theorem]{Remark}

\newtheorem{example}[theorem]{Example}

\title{Laplacian flow of homogeneous $G_2$-structures and its solitons}

\author{Jorge Lauret}

\address{Universidad Nacional de C\'ordoba, FaMAF and CIEM, 5000 C\'ordoba, Argentina}
\email{lauret@famaf.unc.edu.ar}

\thanks{This research was partially supported by grants from CONICET, FONCYT and SeCyT (Universidad Nacional de C\'ordoba)}

\begin{document}

\maketitle

\begin{abstract}
We use the bracket flow/algebraic soliton approach to study the Laplacian flow of $G_2$-structures and its solitons in the homogeneous case.  We prove that any homogeneous Laplacian soliton is equivalent to a semi-algebraic soliton (i.e.\ a $G$-invariant $G_2$-structure on a homogeneous space $G/K$ that flows by pull-back of automorphisms of $G$ up to scaling).  Algebraic solitons are geometrically characterized among Laplacian solitons as those with a `diagonal' evolution.  Unlike the Ricci flow case, where any homogeneous Ricci soliton is isometric to an algebraic soliton, we have found, as an application of the above characterization, an example of a left-invariant closed semi-algebraic soliton on a nilpotent Lie group which is not equivalent to any algebraic soliton.  The (normalized) bracket flow evolution of such a soliton is periodic.  In the context of solvable Lie groups with a codimension-one abelian normal subgroup, we obtain long time existence for any closed Laplacian flow solution; furthermore, the norm of the torsion is strictly decreasing and converges to zero.  We also classify algebraic solitons in this class and exhibit several explicit examples of closed expanding Laplacian solitons.
\end{abstract}

\tableofcontents

\section{Introduction}\label{intro}

The {\it Laplacian flow} for a family $\vp(t)$ of $G_2$-structures on a fixed $7$-dimensional differentiable manifold $M$ is the evolution equation
$$
\dpar\vp(t) = \Delta_{\vp(t)}\vp(t),
$$
where $\Delta_{\vp(t)}$ is the Hodge Laplacian operator on $3$-forms determined by the Riemannian metric $g_{\vp(t)}$ and orientation defined by each $\vp(t)$.  It was introduced back in 1992 by Bryant (see \cite{Bry}) as a tool to try to deform a closed $G_2$-structure to a torsion-free one.  It is well known that torsion-free (or parallel) $G_2$-structures produce Ricci flat Riemannian metrics with holonomy contained in $G_2$.  Foundational results for this flow in the case when $M$ is compact and $\vp$ closed have recently been developed by Lotay-Wei in \cite{Lty} (see also \cite{BryXu,Krg,Grg,Lty2,Lty3}).  The long time behavior of the flow, including long time existence and convergence to a torsion-free $G_2$-structure, is the main natural problem.  In this respect, self-similar solutions play a crucial role in the study of the singularities of the flow.

It follows from the invariance by diffeomorphisms of the flow that a solution $\vp(t)$ starting at a $G_2$-structure $\vp$ will be {\it self-similar}, in the sense that $\vp(t)=c(t)f(t)^*\vp$, for some $c(t)\in\RR^*$ and $f(t)\in\Diff(M)$, if and only if
$$
\Delta_\vp\vp=c\vp+\lca_{X}\vp, \qquad \mbox{for some}\quad c\in\RR, \quad X\in\mathfrak{X}(M)\; \mbox{(complete)}.
$$
In that case, $c(t)=\left(\frac{2}{3}ct+1\right)^{3/2}$ (see \cite[Section 4.4]{BF}), and in analogy to the terminology used in Ricci flow theory, $\vp$ is called a {\it Laplacian soliton} and one says it is {\it expanding}, {\it steady} or {\it shrinking}, if $c>0$, $c=0$ or $c<0$, respectively.  In the compact case,  it was proved in \cite{Lin} that there are no shrinking Laplacian solitons and that the only steady ones are the torsion-free $G_2$-structures (i.e. $\Delta_\vp\vp=0$).  There are some examples of compact and noncompact expanding Laplacian solitons in the {\it coclosed} case (i.e.\ $d\ast\vp=0$) which are all {\it eigenforms}: $\Delta_\vp\vp=c\vp$ for some $c\in\RR$ (see \cite{WssWtt2}, \cite{KrgMcKTsu}) and examples of solitons for the modified Laplacian coflow in \cite{Grg2}.  However, the only compact and closed Laplacian solitons which are eigenforms are the torsion-free $G_2$-structures (see \cite{Lty}).  A closed expanding Laplacian soliton which is not an eigenform was found on a nilpotent Lie group in \cite{BF} (see also \cite{Ncl}).  The existence of compact and closed expanding Laplacian solitons is an open problem, as it is so the existence of noncompact and closed shrinking (or non torsion-free steady) Laplacian solitons.

In this paper, we study the Laplacian flow and its solitons in the homogeneous case.  Our work was motivated by the article \cite{FrnFinMnr} by Fern\'andez-Fino-Manero on nilpotent Lie groups.  They study the existence of left-invariant closed $G_2$-structures yielding a Ricci soliton metric, as well as the Laplacian flow evolution of such structures.

A $7$-manifold endowed with a $G_2$-structure $(M,\vp)$ is said to be {\it homogeneous} if the Lie group of all its symmetries or automorphisms,
$$
\Aut(M,\vp):=\{ f\in\Diff(M):f^*\vp=\vp\}\subset\Iso(M,g_\vp),
$$
acts transitively on $M$.  Each Lie subgroup $G\subset\Aut(M,\vp)$ which is transitive on $M$ gives rise to a presentation of $M$ as a {\it homogeneous space} $G/K$, where $K$ is the isotropy subgroup of $G$ at some point $o\in M$, and $\vp$ becomes a $G$-invariant $G_2$-structure on the homogeneous space $M=G/K$.  In the presence of a {\it reductive} decomposition $\ggo=\kg\oplus\pg$ (i.e.\ $\Ad(K)\pg\subset\pg$) for the homogeneous space $G/K$, every $G$-invariant $G_2$-structure on $G/K$ is determined by a positive $3$-form $\vp$ on $\pg\equiv T_oG/K$ which is $\Ad(K)$-invariant.  By requiring $G$-invariance, the Laplacian flow on $G/K$ becomes equivalent to an ODE on the vector space $(\Lambda^3\pg^*)^K$ and thus short-time existence (forward and backward) and uniqueness (among $G$-invariant ones) of solutions are guaranteed.  We note that if $(M,\vp)$ is homogeneous, then all these $G$-invariant Laplacian flow solutions $\vp(t)$ on $M$ starting at $\vp$ for different transitive groups $G$ must coincide; indeed, such groups are all contained in the full automorphism group $\Aut(M,\vp)$.  Since at the moment the uniqueness of Laplacian flow solutions has only been established in the compact case, we do not know a priori if there are homogeneous solutions other than $\vp(t)$ starting at a noncompact homogeneous $(M,\vp)$.  Such possibility is however considered to be highly unlikely since it is reasonable to expect existence and uniqueness within the class of $G_2$-structures having complete and with bounded curvature associated metrics, as it holds in the Ricci flow case (see \cite{Shi,ChnZhu}).

Given a homogeneous space endowed with a $G$-invariant $G_2$-structure $(G/K,\vp)$, the viewpoint developed in \cite{BF} proposes to evolve the homogeneous space rather than the $3$-form $\vp$.  More precisely, each $(G/K,\vp(t))$, where $\vp(t)$ is the $G$-invariant Laplacian flow solution starting at $\vp$, is replaced by an equivariantly equivalent homogeneous space endowed with an invariant $G_2$-structure
$$
\left(G_{\mu(t)}/K_{\mu(t)},\vp\right),
$$
where $\mu(t)\in\Lambda^2\ggo^*\otimes\ggo$ is the solution to a certain ODE for Lie brackets starting at the Lie bracket $\lb$ of $\ggo$, called the {\it bracket flow}.  Here for each Lie bracket $\mu$ on $\ggo$, $G_\mu$ denotes the simply connected Lie group with Lie algebra $(\ggo,\mu)$ and $K_\mu$ the connected Lie subgroup of $G_\mu$ with Lie algebra $(\kg,\mu|_{\kg\times\kg})$.  The bracket flow has been a useful tool in the study of the Ricci flow and some other curvature flows in the homogeneous case (see e.g.\ \cite[Sections 5.3, 5.5]{BF} for an account of applications by different authors).  The approach strongly uses the following objects determined by the fixed positive $3$-form $\vp$: the $G_2$-invariant decompositions
$$
\glg(\pg)=\ggo_2\oplus\qg, \qquad \qg:=\qg_7\oplus\sym(\pg), \qquad \sog(\pg)=\ggo_2\oplus\qg_7,
$$
and the unique operator $Q_\vp\in\qg$ such that $\theta(Q_\vp)\vp=\Delta_\vp\vp$, where $\theta:\glg(\pg)\longrightarrow\End(\Lambda^3\pg^*)$ is the usual representation (see Remark \ref{Lty-rem} for the relationship between $Q_\vp$ and the operators $\iop_\vp$ and $\jop_\vp$ defined in \cite{Bry}).  It is proved in \cite[Sec.\ 2.2]{Lty} that $Q_\vp\in\sym(\pg)$ if $\vp$ is closed.  The Laplacian flow solution $\vp(t)$ and the bracket flow solution $\mu(t)$ have the same maximal interval of time existence, say $(T_-,T_+)$ with $T_-<0<T_+$.

After some preliminaries on $G_2$ geometry in Section \ref{prelim}, we adapt in Section \ref{LFhom-sec} the machinery developed in \cite{BF} to the Laplacian flow case; the results so obtained include (see Sections \ref{BF-sec} and \ref{sol-sec} for more precise statements):

\begin{itemize}
\item[(i)] The norm $|\Delta_{\vp(t)}\vp(t)|_{\vp(t)}$ of the velocity of the flow must blow up at any finite-time singularity
(compare with \cite[Theorem 1.6]{Lty}).

\item[(ii)] If $\mu(t)$ converge to a Lie bracket $\lambda$, as $t\to T_\pm$, and there is a positive lower bound for the (Lie) injectivity radii of the $G$-invariant metrics $g_{\vp(t)}$ on $G/K$, then $(G_\lambda/K_\lambda,\vp)$ is a Laplacian soliton (possibly non-homeomorphic to $G/K$) and $\left(G/K,\vp(t)\right)$ converge in the pointed (or Cheeger-Gromov) sense to $(G_\lambda/K_\lambda,\vp)$, as $t\to T_\pm$.

\item[(iii)] The following conditions on a simply connected $(G/K,\vp)$ are equivalent:
\begin{itemize}
\item[(a)] The operator $Q_{\vp}$ such that $\theta(Q_{\vp})\vp=\Delta_\vp\vp$ satisfies
$$
Q_{\vp}=cI+D_\pg, \quad\mbox{for some} \; c\in\RR, \; D=\left[\begin{smallmatrix} 0&0\\ 0&D_\pg \end{smallmatrix}\right]\in\Der(\ggo),
$$
i.e.\ $(G/K,\vp)$ is an {\it algebraic soliton}.

\item[(b)] $\Delta_\vp\vp=-3c\vp-\lca_{X_D}\vp$, where $X_D$ denotes the vector field on $G/K$ defined by the one-parameter subgroup of automorphisms of $G$ attached to the derivation $D$ (in particular, $(G/K,\vp)$ is a Laplacian soliton).
\end{itemize}
\end{itemize}
The concept of algebraic soliton has a long and fruitful history in the Ricci flow case, due perhaps to its neat definition as a combination of geometric and algebraic aspects of $(G/K,\vp)$.  It has also been a useful tool to address the existence problem of soliton structures for general curvature flows in almost-hermitian geometry (see \cite{SCF}), the symplectic curvature flow (see \cite{Frn2,SCFmuA}) and the Chern-Ricci flow (see \cite{CRF}).  As in any of these cases, a natural question is how special are algebraic solitons among homogeneous Laplacian solitons.  Unlike the Ricci flow case, where any homogeneous Ricci soliton is isometric to an algebraic soliton (see \cite{Jbl2,ArrLfn2}), we have found a left-invariant closed Laplacian soliton on a nilpotent Lie group which can not be equivalent to any algebraic soliton (see Example \ref{semi-ex}).

General properties of homogeneous Laplacian solitons are studied in Section \ref{homLS}.  Let $G/K$ be a homogeneous space with $G$ simply connected, $K$ connected and compact, and consider the reductive decomposition $\ggo=\kg\oplus\pg$ such that $B(\kg,\pg)=0$, where $B$ is the Killing form of $\ggo$.   We prove that if $\vp$ is a closed $G$-invariant $G_2$-structure, then the following conditions on $(G/K,\vp)$ are equivalent:
\begin{itemize}
\item[(i)] The $G$-invariant Laplacian flow solution starting at $\vp$ is given by
$$
\vp(t)=c(t)f(t)^*\vp,
$$
for some family $f(t)$ of equivariant diffeomorphisms of $G/K$ (i.e. automorphisms of $G$ taking $K$ onto $K$) and $c(t)\in\RR^*$, in which case $(G/K,\vp)$ is called a {\it semi-algebraic soliton}.

\item[(ii)] The operator $Q_{\vp}$ such that $\theta(Q_{\vp})\vp=\Delta_\vp\vp$ satisfies
$$
Q_{\vp}=cI+\tfrac{1}{2}(D_\pg+D_\pg^t), \quad\mbox{for some} \; c\in\RR, \; D=\left[\begin{smallmatrix} 0&0\\ 0&D_\pg \end{smallmatrix}\right]\in\Der(\ggo).
$$

\item[(iii)] The bracket flow solution $\mu(t)$ starting at $\lb$ is given by
$$
\tfrac{\mu(t)}{|\mu(t)|} = \left[\begin{smallmatrix} I&0\\ 0& e^{s(t)A} \end{smallmatrix}\right] \cdot \lb,\qquad  A := \tfrac{1}{2}(D_\pg - D_\pg^t),  \qquad s(t):=-\frac{1}{2c}\log(-2ct+1).
$$
\end{itemize}
In that case, $(G/K,\vp)$ is indeed a Laplacian soliton with $\Delta_\vp\vp=-3c\vp-\lca_{X_D}\vp$.  Note that if in addition in part (ii) one has that $D^t\in\Der(\ggo)$, then $(G/K,\vp)$ is an algebraic soliton and $\tfrac{\mu(t)}{|\mu(t)|} \equiv\lb$.  The following results were also obtained in Section \ref{homLS}:
\begin{itemize}
\item[(a)] Let $(M,\vp)$ be a homogeneous Laplacian soliton and consider $G=\Aut(M,\vp)$.  If the $G$-invariant Laplacian flow solution $\vp(t)$ on $M=G/K$ starting at $\vp$ is self-similar, then $(G/K,\vp)$ is a semi-algebraic soliton.

\item[(b)] A closed semi-algebraic soliton $(G/K,\vp)$ with $K$ compact is {\it Laplacian flow diagonal} (i.e.\ the family of operators $\{ Q_{\vp(t)}:t\in(T_-,T_+)\}$ simultaneously diagonalizes) if and only if it is an algebraic soliton.  Since the condition is an equivalence invariant, this geometrically characterizes algebraic solitons among homogeneous Laplacian solitons.
\end{itemize}

It is worth pointing out that the concepts of semi-algebraic and algebraic solitons are applied to homogeneous spaces, not to homogeneous manifolds.  In this regard, we give an example of a solvable Lie group endowed with a left-invariant closed $G_2$-structure $(G,\vp)$ which is not a semi-algebraic soliton but is  nevertheless equivalent to an algebraic soliton on a different Lie group (see Example \ref{no-sas}).  As an application of the result in part (b) above, we found a left-invariant closed semi-algebraic soliton on a nilpotent Lie group which is not Laplacian flow diagonal and therefore it is not equivalent to any algebraic soliton (see Example \ref{semi-ex}).

It follows from part (iii) above that if the set of nonzero eigenvalues of $A$ is linearly dependent over $\QQ$, then $\tfrac{\mu(t)}{|\mu(t)|}$ is periodic and thus for any expanding (shrinking) semi-algebraic soliton, $\mu(t)$ converges to zero (to infinity) by rounding in a cone as $t\to\infty$ ($t\to\tfrac{1}{2c}$).  If on the contrary, such set is linearly independent over $\QQ$, then $\frac{\mu(t)}{|\mu(t)|}$ is not periodic and develops the following chaotic behavior: each point of the solution is contained in the $\omega$-limit.

In Section \ref{muA-sec}, we work in a more explicit way on the class of {\it almost-abelian} (i.e.\ with a codimension-one abelian normal subgroup) solvable Lie groups.  Left-invariant coclosed and closed $G_2$-structures on these Lie groups have been studied by Freibert in \cite{Frb1,Frb2}.  One attaches to each matrix $A\in\slg(3,\CC)\subset\glg_6(\RR)$ a left-invariant closed $G_2$-structure on a simply connected solvable Lie group denoted by $G_A$.  The Lie algebra of $G_A$ has an orthonormal basis $\{ e_1,\dots,e_{7}\}$ such that $\hg:=\spann\{ e_1,\dots,e_{6}\}$ is an abelian ideal, $\ad{e_{7}}|_{\hg}=A$, and the fixed positive $3$-form is given by
$$
\vp:=e^{127}+e^{347}+e^{567}+e^{135}-e^{146}-e^{236}-e^{245}.
$$
The construction covers, up to equivalence, all left-invariant closed $G_2$-structures on almost abelian Lie groups.  The $G_2$-structure $(G_A,\vp)$ is torsion-free if and only if $A\in\sug(3)$.  After giving some criteria for the equivalence between these structures, we compute their torsion, Ricci curvature and the symmetric operator $Q_A\in\sym(7)$ satisfying $\theta(Q_A)\vp=\Delta_A\vp$ all in terms of $A$, which is actually the only datum that is varying here.  We summarize the main results obtained on this class as follows, after noting that the nonabelian Lie groups of the form $G_A$ which are nilpotent are exactly two and their Lie algebras have been denoted by $\ngo_2$ ($A^2=0$) and $\ngo_6$ ($A^3=0$ and $A^2\ne 0$) in \cite{FrnFinMnr,Ncl}:

\begin{itemize}
\item The Laplacian flow is equivalent to the ODE for $A=A(t)\in\slg(3,\CC)$ given by
$$
\ddt A = -\frac{1}{3}\tr{(A+A^*)^2}A + \unm[A,[A,A^*]] - \unm[A,(A+A^*)^2].
$$
\item Any left-invariant closed Laplacian flow solution $\vp(t)$ on a Lie group $G_A$ is immortal (i.e.\ $T_+=\infty$).  Moreover, the scalar curvature of $g_{\vp(t)}$ is strictly increasing and converges to zero, as $t\to\infty$.

\item The Lie group $G_A$ admits a semi-algebraic soliton if and only if $A$ is either semisimple or nilpotent.  They are all expanding.

\item For $A$ semisimple, such a soliton is algebraic and it is the unique semi-algebraic soliton up to equivalence and scaling among all left-invariant closed $G_2$-structures on $G_A$.

\item  $\ngo_2$ admits only one closed $G_2$-structure up to equivalence and scaling, which is an algebraic soliton.

\item $\ngo_6$ does not have any closed algebraic soliton.  The matrices
$$
A_t:=\left[\begin{smallmatrix}
0&t&0\\
&0&1\\
&&0
\end{smallmatrix}\right]\in\slg(3,\CC), \qquad t>0,
$$
provide a continuous family of $G_2$-structures $(G_{A_t},\vp)$ on $\ngo_6$, and $(G_{A_t},\vp)$ is a semi-algebraic soliton if and only if $t=\sqrt{2}$.
\end{itemize}

\vs \noindent {\it Acknowledgements.} The author is grateful to the referee for very helpful comments.

\section{On $G_2$ geometry}\label{prelim}

Roughly speaking, a $G_2$-structure on a $7$-dimensional differentiable manifold is a smooth identification of each of the tangent spaces with $\Ima\OO$, the imaginary part of the octonions, just as an almost-hermitian structure identifies with $\CC^n$ (endowed with its canonical hermitian inner product) each of the tangent spaces.  In this section, we give a quick overview on $G_2$-structures and refer the reader to \cite{Bry,Krg,Lty} for more detailed treatments.

\subsection{The octonions}\label{oct}
Recall that the octonion algebra $\OO$ is an $8$-dimensional real division algebra which is normed, i.e.\ it admits an inner product $\ip$ such that $|uv|=|u||v|$ for all $u,v\in\OO$.  Analogously to the quaternion numbers $\HH$, the octonion product defines a skew-symmetric bilinear map $\times:\Ima\OO\times\Ima\OO\longrightarrow\Ima\OO$ by $u\times v:=\Ima uv$, which turns to be a {\it cross product}, in the sense that
$$
u\times v\perp u,v, \qquad \mbox{and} \qquad |u\times v|^2=|u|^2|v|^2-\la u,v\ra^2, \qquad\forall u,v\in\Ima\OO.
$$
Curiously enough, all this information can be captured in the $3$-form $\phi$ defined by
$$
\phi(u,v,w):=\la u\times v,w\ra, \qquad\forall u,v,w\in\Ima\OO.
$$
Indeed, the inner product can be recovered from $\phi$ in the following highly non-linear way:
\begin{equation}\label{ip-phi}
\la u,v\ra\vol = \frac{1}{6} i_u(\phi)\wedge i_v(\phi)\wedge\phi,
\end{equation}
where $i_u(\phi)$ is the $2$-form given by $i_u(\phi)(v,w):=\phi(u,v,w)$ and $\vol$ is a nonzero $7$-form, so an orientation is also determined by $\phi$.  The cross product is therefore determined by $\phi$ and the product on $\OO$ is given by $uv=-\la u,v\ra 1+u\times v$.

With respect to a suitable oriented and orthonormal basis $\{ e_1,\dots, e_7\}$ of $\Ima\OO$, the $3$-form $\phi$ is written as
\begin{equation}\label{phi-can}
\phi=e^{123}+e^{145}+e^{167}+e^{246}-e^{257}-e^{347}-e^{356},
\end{equation}
where $e^{ijk}:=e^i\wedge e^j\wedge e^k$ and $\{ e^i\}$ is the dual basis of $\{ e_i\}$.  We identify $\RR^7\equiv\Ima\OO$ from now on by using the basis $\{ e_i\}$.  Another remarkable feature is that the whole process still works under small perturbations of $\phi$: the orbit $\Gl_7(\RR)\cdot\phi$ is open in $\Lambda^3(\RR^7)^*$ and so any $3$-form sufficiently close to $\phi$ is of the form $h\cdot\phi$ for some $h\in\Gl_7(\RR)$, from which one constructs an algebra isomorphic to $\OO$ via $h$.

The automorphism group of $\OO$ is isomorphic to the automorphism group of the cross product $\times$.  This, in turn, coincides with the subgroup of $\Gl_7(\RR)$ stabilizing the $3$-form $\phi$, which is actually contained in $\SO(7)$.  They are all isomorphic to the simply connected and compact $14$-dimensional exceptional simple Lie group $G_2$.  In particular, the openness of the orbit $\Gl_7(\RR)\cdot\phi$ in $\Lambda^3(\RR^7)^*$ follows from dimension count: $49-14=35$.

Using the $\epsilon$-notation introduced in \cite[Section 2.4]{Bry} we rewrite formula \eqref{phi-can} as
$$
\phi=\frac{1}{6}\sum_{i,j,k} \epsilon_{ijk}e^{ijk}, \quad\mbox{or equivalently}, \quad e_i\times e_j=\sum_k\epsilon_{ijk}e_k.
$$
The symbol $\epsilon_{ijk}$ is skew-symmetric in the three indices and satisfies many useful identities.  The Lie algebra $\ggo_2$ of $G_2$ can be described as a subalgebra of $\sog(7)$ as follows:
$$
\ggo_2=\left\{ A=[a_{ij}]\in\sog(7): \sum_{j,k} a_{ij}\epsilon_{ijk}=0, \quad\forall i\right\}.
$$
On the other hand, cross product left-multiplication defines a $7$-dimensional subspace of $\sog(7)$,
\begin{equation}\label{q7}
\qg_7=\left\{ [v_{ij}]:v\in\RR^7,\; v_{ij}:=\sum_k\epsilon_{ijk}\la v,e_k\ra\right\},
\end{equation}
such that $\sog(7)=\ggo_2\oplus\qg_7$.  This is a reductive decomposition for the (symmetric) homogeneous space $\RR P^7=\SO(7)/G_2$.

\subsection{Positive $3$-forms}
Let $\pg$ be a real vector space of dimension $7$.  A $3$-form $\vp\in\Lambda^3\pg^*$ is called {\it positive} if it can be written as in \eqref{phi-can} in terms of some basis, or equivalently, if it belongs to the open orbit $\Gl(\pg)\cdot\phi\subset\Lambda^3\pg^*$.  It follows that the set of all positive $3$-forms is parameterized by the $35$-dimensional homogeneous space $\Gl_7(\RR)/G_2$.  Each positive $3$-form $\vp$ defines a unique inner product $\ip_\vp$ and an orientation via
$$
\la u,v\ra_\vp\vol_\vp = \frac{1}{6} i_u(\vp)\wedge i_v(\vp)\wedge\vp,
$$
where $\vol_\vp:=(\det{h})^{-1}\vol$ if $\vp=h\cdot\phi$ for $h\in\Gl(\pg)$ (see \eqref{ip-phi} and \eqref{phi-can}).  Note that this gives an alternative definition of positivity since the assignment is equivariant in the sense that $\ip_{h\cdot\vp}=h\cdot\ip_\vp$ for any $h\in\Gl(\pg)$.  Thus a Hodge star operator $\ast_\vp:\Lambda^k\pg^*\longrightarrow\Lambda^{7-k}\pg^*$ is also determined by $\vp$ as usual:
$$
\alpha\wedge\ast_\vp\beta = \la\alpha,\beta\ra_\vp\vol_\vp, \qquad\forall\alpha,\beta\in\Lambda^k\pg^*,
$$
where $\ip_\vp$ also denotes the natural inner product defined on $\Lambda\pg^*$ by $\ip_\vp$.  Alternatively,
\begin{equation}\label{Hstar}
\ast_\vp e^{i_1}\wedge\dots\wedge e^{i_k} := \pm e^{j_1}\wedge\dots\wedge e^{j_{7-k}},
\end{equation}
where $\{ e_1,\dots,e_7\}$ is an oriented orthonormal basis of $(\pg^*,\ip_\vp)$ with dual basis $\{ e^i\}$, $\{ i_1,\dots,i_k,j_1,\dots,j_{7-k}\}=\{ 1,\dots, 7\}$ and $e^{i_1}\wedge\dots\wedge e^{i_k}\wedge e^{j_1}\wedge\dots\wedge e^{j_{7-k}}=\pm e^1\wedge\dots\wedge e^7$.  In particular, $\ast_\vp^2=id$.

The $\Gl(\pg)$-orbit of the $4$-form $\ast_\vp\vp$ is also open in $\Lambda^4\pg^*$ and its stabilizer subgroup is isomorphic to $\pm G_2$.  Note that
$$
\ast_\phi\phi = e^{4567}+e^{2367}+e^{2345}+e^{1357}-e^{1346}-e^{1256}-e^{1247}.
$$

Let us fix a positive $3$-form $\vp$ on $\pg$.  Since the orbit $\Gl(\pg)\cdot\vp$ is open in $\Lambda^3\pg^*$, we have that its tangent space at $\vp$ satisfies
\begin{equation}\label{nondeg}
\theta(\glg(\pg))\vp = \Lambda^3\pg^*,
\end{equation}
where $\theta:\glg(\pg)\longrightarrow \End(\Lambda^3\pg^*)$ is the representation obtained as the derivative of the natural left $\Gl(\pg)$-action on $3$-forms $h\cdot\psi=\psi(h^{-1}\cdot,h^{-1}\cdot,h^{-1}\cdot)$, i.e.\
$$
\theta(A)\psi=-\psi(A\cdot,\cdot,\cdot)-\psi(\cdot,A\cdot,\cdot)-\psi(\cdot,\cdot,A\cdot), \qquad\forall A\in\glg(\pg),\quad\psi\in\Lambda^3\pg^*.
$$
The Lie algebra of the stabilizer subgroup $G_2(\vp):=\Gl(\pg)_\vp\simeq G_2$ is given by
$$
\ggo_2(\vp):=\{ A\in\glg(\pg):\theta(A)\vp=0\}\simeq\ggo_2.
$$
We consider the orthogonal complement subspace $\qg(\vp)\subset\glg(\pg)$ of $\ggo_2(\vp)$ relative to the inner product on $\glg(\pg)$ determined by $\ip_\vp$ (i.e.\ $\tr{AB^t}$).  The irreducible $G_2(\vp)$-components of $\qg(\vp)$ are $\qg_1(\vp)=\RR I$, the one-dimensional trivial representation, the ($7$-dimensional) standard representation $\qg_7(\vp)$ (see \eqref{q7}) and $\qg_{27}(\vp)$, the other fundamental representation, which has dimension $27$.  Summarizing, each positive $3$-form $\vp$ determines the following $G_2(\vp)$-invariant decompositions:
\begin{equation}\label{g2-dec}
\begin{array}{c}
\glg(\pg)=\ggo_2(\vp)\oplus\qg(\vp), \qquad \qg(\vp)=\qg_1(\vp)\oplus\qg_7(\vp)\oplus\qg_{27}(\vp), \\ \\
\sog(\pg)=\ggo_2(\vp)\oplus\qg_7(\vp), \qquad \sym(\pg)=\qg_1(\vp)\oplus\qg_{27}(\vp), \qquad \qg_{27}(\vp)=\sym_0(\pg),
\end{array}
\end{equation}
where $\sog(\pg)$ and $\sym(\pg)$ are the spaces of skew-symmetric and symmetric linear maps with respect to $\ip_\vp$, respectively, and $\sym_0(\pg):=\{ A\in\sym(\pg):\tr{A}=0\}$.

It follows from \eqref{nondeg} that $\theta(\qg(\vp))\vp=\Lambda^3\pg^*$; moreover, for every $3$-form $\psi\in\Lambda^3\pg^*$, there exists a unique operator $Q_\psi\in\qg(\vp)$ such that
\begin{equation}\label{nondeg2}
\psi=\theta(Q_\psi)\vp.
\end{equation}

\begin{remark}\label{Lty-rem}
If we identify $\sym(\pg)$ with the space $S^2\pg^*$ of symmetric bilinear forms by using $\ip$, then the linear isomorphism
$
\iop:S^2\pg^*\equiv\sym(\pg)\longrightarrow \Lambda^3_{1}\pg^*\oplus\Lambda^3_{27}\pg^*,
$
defined in \cite[(2.17)]{Bry} (and in \cite[(2.6)]{Lty} with a factor of $1/2$) is given by
\begin{equation}\label{iop}
\iop(A)=-2\theta(A)\vp; \qquad\mbox{in particular}, \quad \iop(Q_\psi)=-2\psi.
\end{equation}
On the other hand, the linear map $\jop:\Lambda^3\pg^*\longrightarrow\sym(\pg)$ defined in \cite[(2.18)]{Bry} (and in \cite{Lty}) satisfies that $\jop(\iop(h))=8h+4\tr(h)\ip$ for any $h\in S^2\pg^*$.  It follows from \eqref{iop} that, in terms of the $Q$-operators, $\jop$ is defined by
\begin{equation}\label{jop}
\jop(\psi)=-2\tr(Q_\psi)I-4Q_\psi, \qquad\forall\psi\in\Lambda^3_{1}\pg^*\oplus\Lambda^3_{27}\pg^*.
\end{equation}
Recall that $\jop$ vanishes on $\Lambda^3_{7}\pg^*$ and it is an isomorphism when restricted to $\Lambda^3_{1}\pg^*\oplus\Lambda^3_{27}\pg^*$.
\end{remark}

\subsection{$G_2$-structures}
A $G_2$-{\it structure} on a $7$-dimensional differentiable manifold $M$ is a differential $3$-form $\vp\in\Omega^3M$ such that $\vp_p$ is positive on $T_pM$ for any $p\in M$, or in other words, $\vp_p$ can be written as in \eqref{phi-can} with respect to some basis $\{ e_1,\dots,e_7\}$ of $T_pM$.  Recall from Section \ref{oct} that this suffices to define an octonion product on each vector space $\RR 1\oplus T_pM$ with $T_pM=\Ima\OO$.  We denote by $g_\vp$ the Riemannian metric on $M$, i.e. $g_\vp(p):=\ip_{\vp_p}$ for all $p\in M$, and by $\ast_\vp:\Omega M\longrightarrow\Omega M$ the Hodge star operator defined by $\vp$.

The presence of a $G_2$-structure on $M$ is equivalent to have a sub-bundle with structure group $G_2$ of the $\Gl_7(\RR)$-frame bundle over $M$, i.e.\ the existence of local frames of $TM$ such that all the transition functions are in $G_2$.  It is well known that a $7$-dimensional manifold $M$ admits a $G_2$-structure if and only if $M$ is orientable and spin.  Two manifolds endowed with $G_2$-structures $(M,\vp)$ and $(M',\vp')$ are called {\it equivalent} if $\vp'=f^*\vp$ for some diffeomorphism $f:M'\longrightarrow M$.

The {\it torsion forms} of a $G_2$-structure $\vp$ on $M$ are the components of the {\it intrinsic torsion} $\nabla_\vp\vp$, where $\nabla_\vp$ is the Levi-Civita connection of $g_\vp$.  They can be defined as the unique differential forms $\tau_i\in\Omega^iM$, $i=0,1,2,3$, such that
\begin{equation}\label{dphi}
d\vp=\tau_0\psi+3\tau_1\wedge\vp+\ast_\vp\tau_3, \qquad d\psi=4\tau_1\wedge\psi+\tau_2\wedge\vp,
\end{equation}
where we set from now on $\psi:=\ast_\vp\vp\in\Omega^4M$.  Let $\Omega^2M=\Omega^2_7M\oplus\Omega^2_{14}M$ and $\Omega^3M=\Omega^3_1M\oplus\Omega^3_7M\oplus\Omega^3_{27}M$ be the decompositions defined by the splitting of the bundles $\Lambda^kM$ into the irreducible $G_2(\vp)$-representations given in \eqref{g2-dec}, where we identify $T_pM\equiv\pg$ via an oriented orthonormal basis.  We have that $\tau_2\in\Omega^2_{14}M$, the space of smooth sections of the subbundle $\ggo_2(\vp)\subset\sog(TM)\equiv\Lambda^2M$, and $\tau_3\in\Omega^3_{27}M$, which corresponds to $\theta(\qg_{27}(\vp))\vp\subset\Lambda^3M$.

Some special classes of $G_2$-structures are defined as follows (see \cite{FrnGry}):
\begin{itemize}
  \item {\it closed} (or {\it calibrated}): $d\vp=0$;
  \item {\it coclosed} (or {\it cocalibrated}): $d\psi=0$;
    \item {\it harmonic}: $\Delta_\vp\vp=0$, where $\Delta_\vp=\ast_\vp d\ast_\vp d-d\ast_\vp d\ast_\vp$ is the Hodge Laplacian operator on $3$-forms;
  \item {\it torsion-free}: $\tau_i=0$, for all $i=0,1,2,3$ (or equivalently, {\it parallel}: $\nabla_\vp\vp=0$);
  \item {\it nearly parallel}: $d\vp=c\psi$ for some nonzero $c\in\RR$, or equivalently, $\tau_i=0$ for all $i=1,2,3$ and $d\tau_0=0$.  In particular, $d\psi=0$ and $\Delta_\vp\vp=c^2\vp$.
\end{itemize}

It was proved in \cite{FrnGry} that the following conditions on a $G_2$-structure $\vp$ are equivalent:
\begin{itemize}
  \item $\vp$ closed and coclosed.
  \item $\vp$ torsion-free.
\end{itemize}
In that case, the holonomy group of $(M,g_\vp)$ is contained in $G_2$, $g_\vp$ is Ricci flat and $(M,\vp)$ is called a $G_2$ {\it manifold}.  In the compact case, $\vp$ harmonic can be added in the list of equivalent conditions above.

\subsection{Laplacian flow}
Any oriented $7$-dimensional Riemannian manifold $(M,g)$ determines a Hodge star operator $\ast:\Omega^k M\longrightarrow\Omega^{7-k}M$ as in \eqref{Hstar}, and also the so called {\it Hodge Laplacian operator} given by
$$
\Delta:\Omega^kM\longrightarrow\Omega^kM, \qquad \Delta:=d^*d+dd^*,
$$
where $d^*:\Omega^{k+1}M\longrightarrow\Omega^kM$, $d^*=(-1)^{k+1}\ast d\ast$, is the adjoint of $d$ (see e.g.\ \cite[7.2]{Ptr}).  Given a $G_2$-structure $\vp$ on $M$, we denote by $\Delta_\vp$ the Hodge Laplacian operator determined by the Riemannian metric $g_\vp$ and orientation defined by $\vp$.  In particular, $\Delta_\vp:\Omega^3M\longrightarrow\Omega^3M$ is given by $\Delta_\vp=\ast_\vp d \ast_\vp d - d \ast_\vp d \ast_\vp$.


The following natural geometric flow for $G_2$-structures was introduced by R. Bryant in \cite{Bry}, and is called the {\it Laplacian flow}:
\begin{equation}\label{LF}
\dpar\vp(t) = \Delta_{\vp(t)}\vp(t),
\end{equation}
where $\vp(t)$ is a one-parameter family of $G_2$-structures on a given $7$-dimensional differentiable manifold $M$.  We refer the reader to the recent article by Lotay and Wei \cite{Lty} and the references therein for further information on this flow.

Let $(T_-,T_+)$ denote the maximal interval of time existence for a Laplacian flow solution $\vp(t)$.  We aim to understand the behavior of $(M,\vp(t))$, as $t$ is approaching a singularity $T_\pm$, in the same spirit as in \cite[Section 3]{Ltt}, where the long-time behavior of homogeneous type-III Ricci flow solutions is studied.  In order to prevent collapsing, the question is whether we can find a manifold endowed with a $G_2$-structure $(M_\pm,\vp_\pm)$, imbeddings $f(t):M_\pm\longrightarrow M$ and a scaling function $a(t)>0$ so that $a(t)f(t)^*\vp(t)$ converges smoothly to $\vp_\pm$, as $t\to T_\pm$.  Sometimes it is only possible to obtain that along a subsequence $t_k\to T_\pm$ and the diffeomorphisms $f(t_k)$ may be only defined on open subsets $\Omega_k$ exhausting $M_\pm$.  Thus $M_\pm$ might be non-homeomorphic to $M$. This is called {\it pointed} or {\it Cheeger-Gromov} convergence of $(M,a(t)\vp(t))$ toward $(M_\pm,\vp_\pm)$ (see \cite[Sec.\ 7]{Lty}).

The following natural questions arise:

\begin{itemize}
\item What is the simplest quantity that, as long as it remains bounded, it prevents the formation of a singularity? (see \cite[Theorem 1.3]{Lty}).

\item Does the scalar curvature of $g_{\vp(t)}$ converge to $-\infty$, as $t\to T_+<\infty$, for any closed Laplacian flow solution $\vp(t)$?  This is equivalent to the blowing up of the torsion at a finite time singularity.
\end{itemize}

\subsection{Laplacian solitons}\label{LS-sec}
A Laplacian flow solution $\vp(t)$ on a differentiable manifold $M$ is called {\it self-similar} if $\vp(t)=c(t)f(t)^*\vp(0)$ for some $c(t)\in\RR^*$ and $f(t)\in\Diff(M)$.  It is well known that the existence of a self-similar solution starting at a $G_2$-structure $\vp$ is equivalent to the following condition:
$$
\Delta_\vp\vp=c\vp+\lca_{X}\vp, \qquad \mbox{for some}\quad c\in\RR, \quad X\in\mathfrak{X}(M)\; \mbox{(complete)},
$$
where $\lca_X$ denotes Lie derivative.  In that case, $c(t)=\left(\frac{2}{3}ct+1\right)^{3/2}$.  In analogy to the terminology used in Ricci flow theory, $\vp$ is called a {\it Laplacian soliton} and one says it is {\it expanding}, {\it steady} or {\it shrinking}, if $c>0$, $c=0$ or $c<0$, respectively.  Note that the maximal interval of existence $(T_-,T_+)$ for these self-similar solutions equals $(-\tfrac{3}{2c},\infty)$, $(-\infty,\infty)$ and $(-\infty,-\tfrac{3}{2c})$, respectively.

Results on Laplacian solitons in the literature include:

\begin{itemize}
\item \cite[Corollary 1]{Lin} There are no compact shrinking Laplacian solitons, and the only compact steady Laplacian solitons are the torsion-free $G_2$-structures (see also \cite[Proposition 9.4]{Lty} for a shorter proof in the closed case).

\item Any nearly parallel $G_2$-structure $\vp$ satisfies $\Delta_\vp\vp=c^2\vp$ and so is a coclosed expanding Laplacian soliton.  Examples are given by the round and squashed spheres (see \cite[Section 4.1]{WssWtt2}).

\item \cite[Section 6]{KrgMcKTsu} Examples of non-compact expanding coclosed Laplacian solitons which are not nearly parallel.  However, they still are all {\it eigenforms} (i.e.\ $\Delta_\vp\vp=c\vp$ for some $c\in\RR$).

\item \cite[Proposition 9.1]{Lty} The only compact and closed Laplacian solitons which are eigenforms are the torsion-free $G_2$-structures.

\item \cite[Section 7]{BF} There is a left-invariant closed $G_2$-structure on a nilpotent Lie group which is an expanding Laplacian soliton and is not an eigenform (see also \cite{Ncl}, where a closed expanding Laplacian soliton has been found on seven of the twelve nilpotent Lie groups admitting a closed $G_2$-structure).
\end{itemize}

The following are natural questions:

\begin{itemize}
\item Are there compact and closed expanding Laplacian solitons?

\item Are there compact expanding Laplacian solitons other than nearly parallel $G_2$-structures?

\item Is any compact Laplacian soliton {\it gradient} (i.e.\ $X=\nabla f$ for some $f\in C^\infty(M)$)?

\item Given a Laplacian flow solution $(M,\vp(t))$, does $(M,a(t)\vp(t))$ converge in the pointed sense to some Laplacian soliton $(M_\pm,\vp_\pm)$, as $t\to T_\pm$, for some scaling function $a(t)>0$?
\end{itemize}

\subsection{Closed $G_2$-structures}
A $G_2$-structure on a $7$-manifold $M$ is closed if and only if the torsion forms $\tau_0$, $\tau_1$ and $\tau_3$ all vanish (see \eqref{dphi}).  Thus only the torsion form $\tau_2$ survives for a closed $G_2$-structure $\vp$.  In that case, the $2$-form $\tau_2$ will be denoted by $\tau_\vp$, and it holds that
\begin{equation}\label{tauphi}
\tau_\vp=-\ast_\vp d\ast_\vp\vp, \qquad d\tau_\vp=\Delta_\vp\vp.
\end{equation}
According to \eqref{nondeg2}, there exists a unique operator $Q_\vp\in\qg(\vp)\subset\End(TM)$ such that
\begin{equation}\label{defQphi}
\theta(Q_\vp)\vp=\Delta_\vp\vp.
\end{equation}
It is proved in \cite[Sec.\ 2.2]{Lty} that $Q_\vp\in\sym(TM)$ for any closed $\vp$ (i.e.\ $\Delta_\vp\vp\in\Omega^3_1M\oplus\Omega^3_{27}M$, or equivalently, its $\Omega^3_{7}M$-component vanishes); moreover, the following formula can be easily deduced from \cite[Proposition 2.2]{Lty} and also from \cite[4.37]{Bry}.  

\begin{proposition}\label{Qphi-form}
For any closed $G_2$-structure $\vp$,
$$
Q_\vp = \Ricci_\vp - \frac{1}{12}\tr\left(\tau_\vp^2\right)I + \unm\tau_\vp^2,
$$
where $\Ricci_\vp$ is the Ricci operator of $(M,g_\vp)$ and $\tau_\vp\in\sog(TM)$ also denotes the skew-symmetric operator determined by the $2$-form $\tau_\vp$ (i.e.\ $\tau_\vp=\la\tau_\vp\cdot,\cdot\ra_\vp$). In particular,
\begin{itemize}
\item[(i)] $|\tau_\vp|^2=-\unm\tr{\tau_\vp^2}$.

\item[(ii)] The scalar curvature of $(M,g_\vp)$ is given by
$$
R(g_\vp)=-\unm|\tau_\vp|^2 = \unc\tr{\tau_\vp^2} = \frac{3}{2}\tr{Q_\vp}.
$$
\item[(iii)] $R(g_\vp)\leq 0$ and it vanishes if and only if $\vp$ is torsion-free.
\end{itemize}
\end{proposition}

\begin{proof}
The formula for $Q_\vp$, which coincides with $-h$ in the notation of \cite{Lty} (see Remark \ref{Lty-rem}), follows from (2.24) and Proposition 2.2 (see also formula (3.4)) in that paper (note also that $T=-\unm\tau_\vp$).  The remaining items follow from \cite[Corollary 2.4]{Lty}.
\end{proof}

\section{Laplacian flow on homogeneous spaces}\label{LFhom-sec}

In this section, we describe an approach developed in \cite{BF} to study geometric flows and their solitons on homogeneous spaces.  We need to state the two main theorems in \cite{BF} in the case of the Laplacian flow.

\subsection{Homogeneous $G_2$-structures and the Laplacian flow}\label{homog-sec}
A $7$-manifold endowed with a $G_2$-structure $(M,\vp)$ is said to be {\it homogeneous} if its automorphism group
$$
\Aut(M,\vp):=\{ f\in\Diff(M):f^*\vp=\vp\},
$$
acts transitively on $M$.  It is known that $\Aut(M,\vp)$ is a Lie group, it is indeed a closed subgroup of the Lie group $\Iso(M,g_\vp)$ of all isometries of the Riemannian manifold $(M,g_\vp)$.  Each Lie subgroup $G\subset\Aut(M,\vp)$ which is transitive on $M$ gives rise to a presentation of $M$ as a {\it homogeneous space} $G/K$, where $K$ is the isotropy subgroup of $G$ at some point $o\in M$, and $\vp$ becomes a $G$-invariant $G_2$-structure on the homogeneous space $M=G/K$. As in the Riemannian case, $G$ is closed in $\Aut(M,\vp)$ if and only if $K$ is compact.  In the presence of a {\it reductive} decomposition $\ggo=\kg\oplus\pg$ (i.e.\ $\Ad(K)\pg\subset\pg$) for the homogeneous space $G/K$, where $\ggo$ and $\kg$ respectively denote the Lie algebras of $G$ and $K$,  every $G$-invariant $G_2$-structure on $G/K$ is determined by a positive $3$-form $\vp$ on $\pg\equiv T_oM$ (the tangent space at the origin $o$ of $G/K$) which is $\Ad(K)$-invariant.  This means that $(\Ad(k)|_\pg)\cdot\vp=\vp$ for any $k\in K$, or equivalently if $K$ is connected, $\theta(\ad{Z}|_\pg)\vp=0$ for all $Z\in\kg$.  Note that the corresponding metric $g_\vp$ is precisely the $G$-invariant metric on $G/K$ whose value at the origin $o$ is $\ip_\vp$.  The existence of a reductive decomposition for $G/K$ is therefore guaranteed, even if $G/K$ is only almost-effective rather than effective as above.  Anyway, one can just work with $\pg=\ggo/\kg$ if a reductive decomposition is preferred not to be chosen.

On a homogeneous space $M=G/K$ with a reductive decomposition $\ggo=\kg\oplus\pg$, if we require $G$-invariance of $\vp(t)$ for all $t$, then the Laplacian flow equation \eqref{LF} becomes equivalent to the ODE for a one-parameter family $\vp(t)$ of $\Ad(K)$-invariant $3$-forms on the single vector space $\pg$ given by
\begin{equation}\label{LF2}
\ddt \vp(t) = \Delta_{\vp(t)}\vp(t),
\end{equation}
where $\Delta_{\vp(t)}:(\Lambda^3\pg^*)^K\longrightarrow(\Lambda^3\pg^*)^K$ is the Hodge Laplacian operator defined by $\vp(t)$ on the space $\Lambda^3\pg^*)^K$ of all $\Ad(K)$-invariant $3$-forms of $\pg$ (i.e.\ $G$-invariant differential $3$-forms of $G/K$).  Indeed, the solutions to \eqref{LF2} are the integral curves of the vector field $X$ on the subspace $(\Lambda^3\pg^*)^K\subset\Lambda^3\pg^*$ defined by $X_\psi:=\Delta_\psi\psi$ for any $\psi\in(\Lambda^3\pg^*)^K$.  The $\Ad(K)$-invariance of the $3$-form $\Delta_\psi\psi$ of $\pg$ follows from the $G$-invariance of the differential $3$-form $\Delta_\psi\psi$ of $M$.

Thus short-time existence (forward and backward) and uniqueness (among $G$-invariant ones) of solutions follow.  Moreover, if $(M,\vp)$ is homogeneous, then all these $G$-invariant Laplacian flow solutions $\vp(t)$ on $M$ starting at $\vp$ for different transitive groups $G$ must coincide as such groups are all contained in the full automorphism group $\Aut(M,\vp)$.  This implies that given $t_0$, since $\Aut(M,\vp)\subset\Aut(M,\vp(t))$ for all $t$, we have that $\vp(t+t_0)$ is the $\Aut(M,\vp)$-invariant solution starting at $\vp(t_0)$.  But thus $\vp(t+t_0)$ is also $\Aut(M,\vp(t_0))$-invariant and so $\Aut(M,\vp(t_0))\subset\Aut(M,\vp)$ by evaluating at $t=-t_0$.

To sum up,

\begin{proposition}\label{LF-exist}
Given a homogeneous $G_2$-structure $(M,\vp)$, there exists a unique Laplacian flow solution $\vp(t)$ on $M$ starting at $\vp$ which is $\Aut(M,\vp)$-invariant.  Furthermore:
\begin{itemize}
\item[(i)] $\vp(t)$ is defined in a maximal interval of time $(T_-,T_+)$, with $T_-<0<T_+$.

\item[(ii)] $\Aut(M,\vp(t))=\Aut(M,\vp)$ for all $t\in(T_-,T_+)$.

\item[(iii)] For any transitive Lie group of automorphisms $G$ of $(M,\vp)$, $\vp(t)$ is the unique $G$-invariant solution on the homogeneous space $M=G/K$.
\end{itemize}
\end{proposition}

Since the uniqueness of Laplacian flow solutions has only been established in the compact case, we do not know a priori if there are homogeneous solutions other than $\vp(t)$ starting at a noncompact homogeneous $(M,\vp)$.  This seems to be very unlikely though.

Two homogeneous spaces endowed with invariant $G_2$-structures $(G/K,\vp)$ and $(G'/K',\vp')$ are called {\it equivariantly equivalent} if there exists an equivariant diffeomorphism $f:G/K\longrightarrow G'/K'$ (i.e. $f$ is determined by a Lie group isomorphism $G\longrightarrow G'$ taking $K$ onto $K'$) such that $\vp=f^*\vp'$.

\subsection{The space of homogeneous spaces}\label{hm}
Let us fix a $(q+7)$-dimensional real vector space $\ggo$ together with a direct sum decomposition
\begin{equation}\label{fixdec}
\ggo=\kg\oplus\pg, \qquad \dim{\kg}=q, \qquad \dim{\pg=7}.
\end{equation}
In order to study invariant $G_2$-structures on homogeneous spaces, we fix in addition a positive $3$-form $\vp$ on $\pg$ and consider the subset $\hca_q:=\hca(\ggo=\kg\oplus\pg,\vp)\subset\Lambda^2\ggo^*\otimes\ggo$ of those skew-symmetric bilinear forms $\mu:\ggo\times\ggo\longrightarrow\ggo$ such that:

\begin{itemize}\label{hqn}
\item [(h1)]  $\mu$ satisfies the Jacobi condition, $\mu(\kg,\kg)\subset\kg$ and $\mu(\kg,\pg)\subset\pg$.
\item[ ]
\item[(h2)] If $G_\mu$ denotes the simply connected Lie group with Lie algebra $(\ggo,\mu)$ and $K_\mu$ is the connected Lie subgroup of $G_\mu$ with Lie algebra $\kg$, then $K_\mu$ is closed in $G_\mu$.
\item[ ]
\item[(h3)] $\{ Z\in\kg:\mu(Z,\pg)=0\}=0$.
\item[ ]
\item[(h4)] $\theta(\ad_{\mu}{Z}|_{\pg})\vp=0$ for all $Z\in\kg$, or equivalently, $\ad_\mu{\kg}|_\pg\subset\ggo_2(\vp)$.
\end{itemize}

It follows that each $\mu\in\hca_q$ defines a unique (almost-effective and simply connected) homogeneous space endowed with an invariant $G_2$-structure,
\begin{equation}\label{hsmu-gamma}
\mu\in\hca_q\rightsquigarrow\left(G_{\mu}/K_{\mu},\vp\right),
\end{equation}
with reductive decomposition $\ggo=\kg\oplus\pg$.  Conversely, any homogeneous $G_2$-structure is equivariantly equivalent to some $\mu\in\hca_q$, for some $q$, up to covering.

For $q=0$, conditions (h2)-(h4) trivially hold and so $\hca_0$ is simply the variety $\lca$ of $7$-dimensional Lie algebras.  We are therefore identifying each $\mu\in\lca$ with $(G_\mu,\vp)$, the simply connected Lie group $G_\mu$ endowed with the left-invariant $G_2$-structure determined by the fixed positive $3$-form $\vp$ we have on the Lie algebra $(\ggo,\mu)$ of $G_\mu$.

\subsection{Bracket flow}\label{BF-sec}
For each $\mu\in\hca_q$, the Hodge Laplacian of $\left(G_{\mu}/K_{\mu},\vp\right)$ is given by
$$
\Delta_\mu:(\Lambda^3\pg^*)^{K_\mu}\longrightarrow(\Lambda^3\pg^*)^{K_\mu}, \qquad \Delta_\mu:= \ast d_\mu\ast d_\mu-d_\mu\ast d_\mu\ast,
$$
where $d_\mu$ is the differential operator on the manifold $G_\mu/K_\mu$ and $\ast$ denotes the fixed Hodge star operator defined by $\vp$ on $\Lambda^3\pg^*$.
\begin{example}\label{n2-1}
For each $a,b,c,d\in\RR$, consider the $7$-dimensional nilpotent Lie algebra $\ggo$ with basis $\{ e_1,\dots,e_7\}$ and Lie bracket $\mu=\mu_{a,b,c,d}\in\hca_0=\lca$ defined by
$$
\left\{\begin{array}{l}
\mu(e_1,e_2)=-ae_5-be_6, \\ \mu(e_1,e_3)=-ce_5-de_6;
\end{array}\right. \quad\mbox{or equivalently}, \quad
\left\{\begin{array}{l}
d_\mu e^5=ae^{12}+ce^{13}, \\ d_\mu e^6=be^{12}+de^{13}.
\end{array}\right.
$$
The $3$-form
$$
\vp=e^{147}+e^{267}+e^{357}+e^{123}+e^{156}+e^{245}-e^{346},
$$
is positive and so it determines a left-invariant $G_2$-structure $\vp$ on the simply connected Lie group $G_\mu$ with Lie algebra $(\ggo,\mu)$.  It is easy to check that $d_\mu\vp=(d-a)e^{1237}-(b+c)e^{1234}$, which implies that $\vp$ is closed if and only if $d=a$ and $c=-b$.  Thus for a closed $G_2$-structure $(G_\mu,\vp)$ we have that
\begin{align}
\ast\vp =& e^{2356}-e^{1345}-e^{1246}+e^{4567}+e^{2347}-e^{1367}+e^{1257}, \notag \\
d_\mu\ast\vp =& -ae^{12467}+be^{13467}+be^{12457}+ae^{13457}, \notag \\
\ast d_\mu\ast\vp =& ae^{35}+be^{25}+be^{36}-ae^{26}, \label{n2-1eq}\\
d_\mu\ast d_\mu\ast\vp =& -2(a^2+b^2)e^{123}, \notag
\end{align}
and so $\Delta_\mu\vp=2(a^2+b^2)e^{123}$.
\end{example}

We denote by $Q_\mu$ the operator $Q_{\Delta_\mu\vp}\in\qg(\vp)\subset\glg(\pg)$ defined by \eqref{nondeg2}, i.e.\
$$
\theta(Q_\mu)\vp=\Delta_\mu\vp.
$$

\begin{example}\label{n2-2}
It follows from \eqref{n2-1eq} that the torsion $2$-form of the closed $G_2$-structures $(G_\mu,\vp)$ in the above example is given by $\tau_\mu= -ae^{35}-be^{25}-be^{36}+ae^{26}$, and as a matrix by
$$
\tau_\mu=\left[\begin{smallmatrix}
0&&&&&&\\
&0&&&b&-a&\\
&&0&&a&b&\\
&&&0&&&\\
&-b&-a&&0&&\\
&a&-b&&&0&\\
&&&&&&0
\end{smallmatrix}\right].
$$
Thus $\tau_\mu^2=(a^2+b^2)\Diag(0,-1,-1,0,-1,-1,0)$, and since the Ricci operator of $(G_\mu,\ip_\vp)$ equals
$$
\Ricci_\mu=(a^2+b^2)\Diag(-1,-\tfrac{1}{2},-\tfrac{1}{2},0,\tfrac{1}{2},\tfrac{1}{2},0),
$$
we obtain from Proposition \ref{Qphi-form} that the operator $Q_\mu = \Ricci_\mu - \frac{1}{12}\tr\left(\tau_\mu^2\right)I + \unm\tau_\mu^2$ is given by
$$
Q_\mu=(a^2+b^2)\Diag(-\tfrac{2}{3},-\tfrac{2}{3},-\tfrac{2}{3},\tfrac{1}{3},\tfrac{1}{3},\tfrac{1}{3},\tfrac{1}{3}).
$$
\end{example}

Consider the following evolution equation for a family $\mu(t)\in \Lambda^2\ggo^*\otimes\ggo$ of brackets, called the {\it bracket flow}:
\begin{equation}\label{BF}
\ddt\mu(t)=\delta_{\mu(t)}\left(\left[\begin{smallmatrix} 0&0\\ 0&Q_{\mu(t)} \end{smallmatrix}\right]\right), \qquad\mu(0)=\lb,
\end{equation}
where $\delta_\mu:\glg(\ggo)\longrightarrow\Lambda^2\ggo^*\otimes\ggo$ is defined in terms of the derivative of the $\Gl(\ggo)$-action $h\cdot\mu:=h\mu(h^{-1}\cdot,h^{-1}\cdot)$ by
\begin{equation}\label{delta}
\delta_\mu(E):=\mu(E\cdot,\cdot)+\mu(\cdot,E\cdot)-E\mu(\cdot,\cdot), \qquad\forall E\in\glg(\ggo).
\end{equation}
The set $\hca_q$ is invariant under the bracket flow and only $\mu(t)|_{\pg\times\pg}$ is actually evolving (see \cite[Lemma 1]{BF}).

\begin{example}\label{n2-3}
For $(G_\mu,\vp)$ in Examples \ref{n2-1} and \ref{n2-2}, it is straightforward to check that $\delta_\mu(Q_\mu)=-\frac{5}{3}(a^2+b^2)\mu$, so the bracket flow equation is given by $\ddt\mu(t)=-\frac{5}{12}|\mu(t)|^2\mu(t)$ and the bracket flow solution by $\mu(t)=c(t)\mu(0)$, for some positive strictly decreasing function $c(t)$ defined on $(T_-,\infty)$, $-\infty<T_-$, such that $c(t)\to 0$, as $t\to\infty$.
\end{example}

We are now ready to state the first main result from \cite{BF} applied to $G_2$-structures.  Let $(G/K,\vp)$ be a simply connected homogeneous space (assume $G$ simply connected and $K$ connected) endowed with a $G$-invariant $G_2$-structure $\vp$ and a reductive decomposition $\ggo=\kg\oplus\pg$. We consider the one-parameter families
$$
(G/K,\vp(t)), \qquad \left(G_{\mu(t)}/K_{\mu(t)},\vp\right),
$$
where $\vp(t)$ is the solution to the Laplacian flow \eqref{LF2} starting at $\vp$ and $\mu(t)$ is the solution to the bracket flow \eqref{BF} starting at the Lie bracket $\lb$ of $\ggo$, the Lie algebra of $G$.  Note that $\ggo=\kg\oplus\pg$ is a reductive decomposition for each of the homogeneous spaces involved.

\begin{theorem}\label{BF-thm}\cite[Theorem 5]{BF}
There exist equivariant diffeomorphisms $f(t):G/K\longrightarrow G_{\mu(t)}/K_{\mu(t)}$ such that
$$
\vp(t)=f(t)^*\vp, \qquad\forall t\in(T_-,T_+),
$$
i.e.\ $(G/K,\vp(t))$ and $\left(G_{\mu(t)}/K_{\mu(t)},\vp\right)$ are equivariantly equivalent.  Moreover, each $f(t)$ can be chosen to be the equivariant diffeomorphism defined by the Lie group isomorphism $G\longrightarrow G_{\mu(t)}$ with derivative $\tilde{h}(t):=\left[\begin{smallmatrix} I&0\\ 0&h(t) \end{smallmatrix}\right]:\ggo\longrightarrow\ggo$, where $h(t):=df(t)|_o:\pg\longrightarrow\pg$ is the solution to any of the following ODE's:
\begin{itemize}
\item[(i)] $\ddt h(t)=-h(t)Q_{\vp(t)}$, $h(0)=I$, where $Q_{\vp(t)}\in\qg_{\vp(t)}$ satisfies $\theta(Q_{\vp(t)})\vp(t)=\Delta_{\vp(t)}\vp(t)$.
\item[(ii)] $\ddt h(t)=-Q_{\mu(t)} h(t)$, $h(0)=I$, where $Q_\mu\in\qg_{\vp}$ satisfies $\theta(Q_\mu)\vp=\Delta_\mu\vp$.
\end{itemize}
The following conditions also hold:
$$
\vp(t)=h(t)^*\vp=h(t)^{-1}\cdot\vp, \qquad \mu(t)=\tilde{h}(t)\cdot\lb.
$$
\end{theorem}

A direct consequence of the theorem is that the maximal interval of time $(T_-,T_+)$ where a solution exists is the same for both flows, so the bracket flow can be used as a tool to study regularity questions on the Laplacian flow (see Section \ref{BFA-sec}).  It is proved for example in \cite[Proposition 4]{BF} that the norm $|\Delta_{\vp(t)}\vp(t)|_{\vp(t)}$ of the velocity of the flow must blow up at a finite-time singularity.  This has been proved in \cite[Theorem 1.6]{Lty} in the case when $M$ is compact and $\vp$ is closed.

The scaling $(G_\mu/K_\mu,c^{-3}\vp)$ is equivariantly equivalent to the element $c\cdot\mu\in\hca_{q,n}$, defined by,
\begin{equation}\label{scmu}
c\cdot\mu|_{\kg\times\kg}=\mu, \qquad c\cdot\mu|_{\kg\times\pg}=\mu, \qquad c\cdot\mu|_{\pg\times\pg}=c^2\mu_{\kg}+c\mu_{\pg},
\end{equation}
where the subscripts denote the $\kg$- and $\pg$-components of $\mu|_{\pg\times\pg}$ given by
\begin{equation}\label{decmu}
\mu(X,Y)=\mu_{\kg}(X,Y)+\mu_{\pg}(X,Y), \qquad \mu_{\kg}(X,Y)\in\kg, \quad \mu_{\pg}(X,Y)\in\pg, \quad\forall X,Y\in\pg.
\end{equation}
Note that $\ip_{c^{-3}\vp}=c^{-2}\ip_\vp$.  The $\RR^*$-action on $\hca_{q,n}$, $\mu\mapsto c\cdot\mu$, may therefore be considered as a {\it geometric scaling} of $(G_\mu/K_\mu,\vp)$ (see \cite[(23)]{BF}).

The previous theorem has the following application on convergence, which follows from \cite[Corollary 4]{BF} and \cite[Theorem 7.1]{Lty}.

\begin{corollary}\label{i-conv}
Assume that $c_k\cdot\mu(t_k)\to\lambda\in\hca_{q,n}$ for some subsequence of times $t_k\to T_\pm$ and numbers $c_k\ne 0$.
\begin{itemize}
\item[(i)] If there is a positive lower bound for the Lie injectivity radii (see \cite[Definition 3]{BF}) of the $G$-invariant metrics $c_k^{-2}g_{\vp(t_k)}$ on $G/K$, then, after possibly passing to a subsequence, $\left(G/K,c_k^{-3}\vp(t_k)\right)$ converge in the pointed (or Cheeger-Gromov) sense to $(G_\lambda/K_\lambda,\vp)$, as $k\to\infty$.

\item[(ii)] In the case of a Lie group $G$ (i.e.\ $K$ trivial), the hypothesis on the Lie injectivity radii in part (i) can be removed.  Moreover, if either $G_\lambda$ is compact or $G$ is completely solvable, then $(G,c_k^{-3}\vp(t_k))$ smoothly converges up to pull-back by diffeomorphisms to $(G_\lambda,\vp)$, as $k\to\infty$.
\end{itemize}
\end{corollary}

We note that the limiting Lie group $G_\lambda$ in the above corollary might be non-isomorphic to $G$, and consequently in part (i), the limiting homogeneous space $G_\lambda/K_\lambda$ might be non-homeomorphic to $G/K$.  If we strength the hypothesis in the above corollary to $c(t)\cdot\mu(t)\to\lambda$, as $t\to T_\pm$, for some smooth function $c(t)\in\RR^*$, then the limit $(G_\lambda/K_\lambda,\vp)$ is in fact a Laplacian soliton (see \cite[Proposition 4.1]{homRS}).

\begin{example}
It follows from Corollary \ref{i-conv}, (ii) that for the bracket flow solution given in Example \ref{n2-3} and $\mu_0:=\mu_{1,0,0,1}$, one obtains that the Laplacian flow solution $(G_{\mu_0},\vp(t))$ smoothly converges up to pull-back by diffeomorphisms to the flat $(\RR^7,\vp)$, as $t\to\infty$.
\end{example}

The bracket flow has been a useful tool in the study of the Ricci flow and some other curvature flows in the homogeneous case, see \cite[Sections 5.3, 5.5]{BF} for a quick overview.

\subsection{Algebraic solitons}\label{sol-sec}
On homogeneous spaces, it is natural to study the existence of the following Laplacian solitons of an `algebraic' nature.

\begin{theorem}\label{rsequiv}\cite[Theorem 6]{BF}
For a homogeneous space $(G/K,\vp)$ endowed with a $G$-invariant $G_2$-structure $\vp$, with $G$ simply connected and $K$ connected, the following conditions are equivalent:
\begin{itemize}
\item[(i)] The bracket flow solution starting at $\lb$ is given by
$$
\mu(t)=c(t)\cdot\lb, \qquad\mbox{for some}\quad c(t)>0, \quad c(0)=1.
$$
\item[(ii)] The operator $Q_{\vp}\in\qg_\vp$ such that $\theta(Q_{\vp})\vp=\Delta_\vp\vp$ satisfies
\begin{equation}\label{as}
Q_{\vp}=cI+D_\pg, \qquad \mbox{for some} \quad c\in\RR, \quad D=\left[\begin{smallmatrix} 0&0\\ 0&D_\pg \end{smallmatrix}\right] \in\Der(\ggo).
\end{equation}
\end{itemize}
In that case, $(G/K,\vp)$ is a Laplacian soliton with
$$
\Delta_\vp\vp=-3c\vp-\lca_{X_D}\vp,
$$
where $X_D$ denotes the vector field on $G/K$ defined by the one-parameter subgroup of $\Aut(G)$ attached to the derivation $D$, the scaling function in (i) is
$$
c(t)=(-2ct+1)^{-1/2},
$$
and the $G$-invariant Laplacian flow solution starting at $\vp$ is given by
\begin{equation}\label{solgama}
\vp(t)=b(t)e^{s(t)D_\pg}\cdot\vp, \qquad b(t):=(-2ct+1)^{3/2}, \qquad s(t):=-\frac{1}{2c}\log(-2ct+1).
\end{equation}
(For $c=0$, set $s(t)=t$).
\end{theorem}

\begin{definition}
A homogeneous space $(G/K,\vp)$ endowed with a $G$-invariant $G_2$-structure and a reductive decomposition $\ggo=\kg\oplus\pg$ is called an {\it algebraic soliton} if condition \eqref{as} holds.
\end{definition}

It follows from the above theorem that any simply connected algebraic soliton is indeed a Laplacian soliton.  The concept of algebraic soliton has been very fruitful in the study of homogeneous Ricci solitons since its introduction in \cite{soliton}, we refer to \cite[Sections 5.2, 5.4]{BF} for a quick overview (see also \cite{SCF,Frn2,SCFmuA,CRF} for other curvature flows).  Nothing changes by allowing a derivation of the form $D=\left[\begin{smallmatrix} \ast&0\\ 0&D_\pg \end{smallmatrix}\right]\in\Der(\ggo)$ in the definition of algebraic soliton since $D\kg=0$ must necessarily holds (see \cite[Remark 7]{BF}).  It is proved in \cite[Section 4]{homRS} (see also \cite{ArrLfn2}) that algebraic solitons are precisely the fixed points, and hence the possible limits of any normalized bracket flow.  Furthermore, given a starting point, one can obtain at most one non-flat algebraic soliton as a limit by running all possible normalized bracket flow solutions (see Corollary \ref{i-conv}).

\begin{example}\label{n2-sol}
According to Theorem \ref{rsequiv}, the closed $G_2$-structures $(G_\mu,\vp)$, $\mu=\mu_{a,b,a,-b}$, studied in Example \ref{n2-2} are all expanding algebraic solitons.  Indeed,
$$
Q_\mu=(a^2+b^2)\Diag(-\tfrac{2}{3},-\tfrac{2}{3},-\tfrac{2}{3},\tfrac{1}{3},\tfrac{1}{3},\tfrac{1}{3},\tfrac{1}{3}) = (a^2+b^2)(-\tfrac{5}{3}I+D),
$$
where $D=\Diag(1,1,1,2,2,2,2)\in\Der(\mu)$.   It is easy to see that they are all pairwise equivalent up to scaling and that $(G_\mu,\ip_\vp)$ is a Ricci soliton for any $a,b$ with Ricci operator
$$
\Ricci_\mu=(a^2+b^2)\left(-2I+\Diag(1,\tfrac{3}{2},\tfrac{3}{2},2,\tfrac{5}{2},\tfrac{5}{2},2)\right)\in\RR I+\Der(\mu).
$$
\end{example}

\section{Homogeneous Laplacian solitons}\label{homLS}

Due to the lack of a general uniqueness result, given a Laplacian soliton $(M,\vp)$ which is homogeneous, one can not assert that the corresponding self-similar Laplacian flow solution starting at $\vp$ (see Section \ref{LS-sec}) coincide with the $\Aut(M,\vp)$-invariant solution $\vp(t)$ given in Proposition \ref{LF-exist}.  In particular, we do not know if $\vp(t)$ is self-similar.

A diffeomorphism $f$ of a homogeneous space $G/K$ is said to be {\it equivariant} if
$$
f(hK)=\tilde{f}(h)K, \qquad\forall h\in G, \qquad\mbox{for some} \quad \tilde{f}\in\Aut(G), \quad \tilde{f}(K)=K.
$$
Let $\Aut(G/K)$ denote the subgroup of $\Diff(G/K)$ of all equivariant diffeomorphisms of $G/K$.

\subsection{Semi-algebraic solitons}
Given a simply connected algebraic soliton $(G/K,\vp)$, so a Laplacian soliton, it follows from \eqref{solgama} that the diffeomorphisms $f(t)$ defining the corresponding self-similar solution starting at $\vp$ can be taken as the equivariant diffeomorphisms defined by the automorphisms of $G$ with derivatives $e^{-s(t)D}\in\Aut(\ggo)$.  This motivates the following more general way to consider a Laplacian soliton to be `algebraic' (see \cite[Section 2]{Jbl} and \cite[Section 3]{homRS} for the Ricci flow case).

\begin{definition}\label{sas-def}
A homogeneous space endowed with a $G$-invariant $G_2$-structure $(G/K,\vp)$ is called a {\it semi-algebraic soliton} if there exists a one-parameter family $f(t)\in\Aut(G/K)$ such that the $G$-invariant Laplacian flow solution starting at $\vp$ is given by $\vp(t)=c(t)f(t)^*\vp$, for some $c(t)\in\RR^*$.
\end{definition}

In other words, a semi-algebraic soliton $(G/K,\vp)$ is a Laplacian soliton for which the solution $\vp(t)$ stays equivariantly equivalent to $\vp$ for all $t$ rather that only equivalent.  The following result was proved in \cite[Theorem 3.1]{Jbl} for Ricci solitons; we essentially follow the lines of that proof.

\begin{proposition}\label{sas-hom}
Let $(M,\vp)$ be a homogeneous Laplacian soliton and consider
$$
G=\Aut(M,\vp).
$$
If the $G$-invariant Laplacian flow solution $\vp(t)$ on $M=G/K$ starting at $\vp$ is self-similar, then $(G/K,\vp)$ is a semi-algebraic soliton.  Moreover, there exist $f(t)\in\Aut(G/K)$, $f(0)=id$, such that,
\begin{itemize}
\item[(i)] $\vp(t)=c(t)f(t)^*\vp$, for some $c(t)\in\RR^*$.
\item[(ii)] $f(t)|_K=id$ for all $t$.
\end{itemize}
\end{proposition}

\begin{proof}
Consider the presentation $M=G/K$, where $K$ is the isotropy subgroup of $G$ at some $p\in M$.  If $\vp(t)=c(t)g(t)^*\vp$, for $c(t)\in\RR^*$, $g(t)\in\Diff(M)$, then we can assume that $g(0)=id$, and also that $g(t)(p)=p$ for all $t$ by composing with $h(t)\in G$ such that $h(t)^{-1}(p)=g(t)(p)$.  Since $G=\Aut(M,\vp(t))=g(t)^{-1}Gg(t)$ for all $t$ (see Proposition \ref{LF-exist}, (ii)), we can define an isomorphism $\tilde{f}(t):G\longrightarrow G$ by $\tilde{f}(t)(h):=g(t)hg(t)^{-1}$ for all $h\in G$, which in turns determines $f(t)\in\Aut(G/K)$ as $\tilde{f}(t)(K)\subset K$.  But $f(t)=g(t)$ on $G/K$ for all $t$:
$$
f(t)(hK) = g(t)hg(t)^{-1}K = g(t)hg(t)^{-1}(p) = g(t)(h(p)) = g(t)(hK), \qquad\forall h\in G,
$$
and so part (i) follows.

To prove that part (ii) can also be assumed to hold, we first note that $K$ is compact and so the identity component $\Aut(K)_0$ of $\Aut(K)$ consists of inner automorphisms.  Since $\tilde{f}(t)|_K$ is continuous and $\tilde{f}(0)=id$, we have that $\tilde{f}(t)|_K\in\Aut(K)_0$ for all $t$, so there exist $k(t)\in K$, $k(0)=e$, such that $\tilde{f}(t)|_K=I_{k(t)}$, where $I_k$ denotes conjugation by $k$.  Now the equivariant diffeomorphism $\tau(k(t))$ of $G/K$ determined by $I_{k(t)}\in\Aut(G)$ belongs to $\Aut(M,\vp)$ (recall that $dI_{k}|_o=Ad(k)$ for any $k\in K$), so parts (i) and (ii) both hold for the equivariant diffeomorphisms $\tau(k(t))^{-1}f(t)$, concluding the proof of the proposition.
\end{proof}

\begin{remark}\label{fmore}
The following properties can be deduced from the proof of the above proposition:
\begin{itemize}
\item[(i)] If $\vp(t)=c(t)f(t)^*\vp$, for some $c(t)\in\RR^*$ and $f(t)\in\Diff(M)$ such that $f(0)=id$, $f(t)(o)=o$, then $f(t)\in\Aut(G/K)$ for all $t$ for $G=\Aut(M,\vp)$.
\item[(ii)] The fact that $f(t)|_K=id$ for all $t$ can be assumed for any semi-algebraic soliton $(G/K,\vp)$ with $K$ compact.
\end{itemize}
\end{remark}

\begin{corollary}
Assume that for each homogeneous $G_2$-structure $(M,\vp)$, there exists a unique Laplacian flow solution $(M,\vp(t))$ starting at $\vp$ such that $(M,\vp(t))$ is homogeneous for all $t$.  Then any homogeneous Laplacian soliton $(M,\vp)$ is a semi-algebraic soliton when presented as a homogeneous space $(G/K,\vp)$ with $G=\Aut(M,\vp)$.
\end{corollary}

We now give a characterization of semi-algebraic solitons based on a formula for the operator $Q_\vp$ in terms of derivations, as we have for algebraic solitons (see \eqref{as}), which provides a useful tool to study the existence, uniqueness and structure of homogeneous Laplacian solitons.

\begin{proposition}\label{sas-der}
If $(G/K,\vp)$ is a semi-algebraic soliton, then for any reductive decomposition $\ggo=\kg\oplus\pg$ for $G/K$,
\begin{equation}\label{alg2}
Q_{\vp}=cI+\proy_{\vp}(D_{\pg}), \qquad\mbox{for some} \quad c\in\RR, \quad
D=\left[\begin{smallmatrix} \ast&\ast\\ 0&D_\pg \end{smallmatrix}\right]\in\Der(\ggo),
\end{equation}
where $\proy_{\vp}:\glg(\pg)=\ggo_2(\vp)\oplus\qg(\vp)\longrightarrow\qg(\vp)$ is the usual linear projection.  The converse holds for $G$ simply connected and $K$ connected, and in that case, $\Delta_\vp\vp=-3c\vp - \lca_{X_D}\vp$ and the Laplacian flow solution starting at $\vp$ is given by $\vp(t)=b(t)e^{s(t)D_\pg}\cdot\vp$ as in \eqref{solgama}.
\end{proposition}

\begin{remark}\label{sas-closed}
In the case when $\vp$ is closed, one has that $Q_\vp\in\sym(\pg)=\qg_1(\vp)\oplus\qg_{27}(\vp)$ (see Lemma \ref{Qphi-form}) and so condition \eqref{alg2} becomes
$$
Q_\vp=cI+\unm\left(D_\pg+D_\pg^t\right).
$$
\end{remark}

\begin{proof}
If $(G/K,\vp)$ is a semi-algebraic soliton, then $\vp(t)=a(t)f(t)^*\vp$, for some $a(t)\in\RR^*$, $a(0)=1$, $f(t)\in\Aut(G/K)$, $f(0)=id$.  By taking the derivative at $t=0$, one obtains that $\Delta_\vp\vp=a\vp+\lca_X\vp$ for $a=a'(0)$ and $X$ the vector field on $G/K$ defined by $X_p:=\ddt|_0f(t)(p)$.  Now consider $$
D=\left[\begin{smallmatrix} \ast&\ast\\ 0&D_\pg \end{smallmatrix}\right]=-\ddt\Big|_0d\tilde{f}(t)|_e\in\Der(\ggo),
$$
where for each $t$, $\tilde{f}(t)\in\Aut(G)$ is the automorphism defining $f(t)$.  It follows that
\begin{equation}\label{derlie}
\lca_X\vp(o) = \ddt\Big|_0 (f(t)^*\vp)(o) = \ddt\Big|_0 df(t)|_o^*\vp = -\theta\left(\ddt\Big|_0 df(t)|_o\right)\vp = \theta(D_\pg)\vp.
\end{equation}
The formula $\Delta_\vp\vp=a\vp+\lca_X\vp$ therefore becomes $\theta(Q_\vp)\vp=\theta\left(-\frac{1}{3}aI+D_\pg\right)\vp$ when evaluated at the origin $o$, from which condition \eqref{alg2} follows for $c=-\frac{1}{3}a$.

Conversely, assume that condition (\ref{alg2}) holds for some reductive decomposition $\ggo=\kg\oplus\pg$ for $G/K$.  Since $D\in\Der(\ggo)$, $e^{tD}\in\Aut(\ggo)$ for all $t\in\RR$ and since $G$ is simply connected, there exists $\tilde{f}(t)\in\Aut(G)$ such that $d\tilde{f}(t)|_e=e^{tD}$.  By using that $K$ is connected and $D\kg=\kg$, we obtain for each $t$ that $\tilde{f}(t)(K)=K$, so $\tilde{f}(t)$ defines $f(t)\in\Aut(G/K)$, for which $df(t)|_{o}=e^{tD_{\pg}}$.  Let  $X_D$ denote the vector field on $G/K$ defined by $X_D(p)=\ddt|_0 f(t)(p)$.  It follows as in \eqref{derlie} that $\lca_{X_D}\vp(o) = -\theta(D_\pg)\vp$, so from \eqref{alg2} we obtain that
$$
\Delta_\vp\vp = \theta(Q_\vp)\vp = c\theta(I)\vp + \theta(D_\pg)\vp = -3c\vp - \lca_{X_D}\vp,
$$
by using that every tensor in this formula is $G$-invariant (recall that the flow of $X_D$ is given by automorphisms of $G$).  It is easy to see that this implies that the Laplacian flow solution is given by $\vp(t)=b(t)f(-s(t))^*\vp$ for $b(t)$ and $s(t)$ as in \eqref{solgama}, concluding the proof of the proposition.
\end{proof}

\begin{remark}\label{Dmore}
One can obtain a derivation $D$ with a simpler structure in the above proposition as follows:
\begin{itemize}
\item[(i)] If the reductive decomposition $\ggo=\kg\oplus\pg$ for $G/K$ with $B(\kg,\pg)=0$ is considered, where $B$ is the Killing form of $\ggo$, then $D\pg\subset\pg$ (see \cite[Lemma 3.10]{homRS}).

\item[(ii)] It follows from Remark \ref{fmore}, (ii) that if $K$ is compact, then $D\kg=0$.
\end{itemize}
\end{remark}

\subsection{Bracket flow evolution of semi-algebraic solitons}
In Theorem \ref{rsequiv}, algebraic solitons have been characterized as the $G_2$-structures that evolves as simply as possible along the bracket flow.  It is then natural to ask about the bracket flow evolution of semi-algebraic solitons.  Let $(G/K,\vp)$ be a closed semi-algebraic soliton with $K$ compact and consider the reductive decomposition $\ggo=\kg\oplus\pg$ with $B(\kg,\pg)=0$.  From the above proposition and Remarks \ref{sas-closed} and \ref{Dmore}, we obtain that
\begin{equation}\label{sas-D}
Q_{\vp}=cI+\unm\left(D_{\pg}+D_\pg^t\right), \qquad\mbox{for some} \quad c\in\RR, \quad
D=\left[\begin{smallmatrix} 0&0\\ 0&D_\pg \end{smallmatrix}\right]\in\Der(\ggo).
\end{equation}
In that case, the formula for the bracket flow solution starting at the Lie bracket $\lb$ of $\ggo$ has been computed in \cite[Proposition 4.2 and Remark 4.3]{homRS} and is given by
\begin{equation}\label{saBF2}
\mu(t) = (-2ct+1)^{-1/2} \cdot \left(\left[\begin{smallmatrix} I&0\\ 0& e^{s(t)A} \end{smallmatrix}\right] \cdot \lb\right),\qquad  A := \unm(D_\pg - D_\pg^t),
\end{equation}
where $s(t)$ is as in \eqref{solgama}.  Note that if in addition, $D^t\in\Der(\ggo)$, then we recover the formula for $\mu(t)$ in the case of an algebraic soliton given in Theorem \ref{rsequiv}, (i).

We note that for expanding solitons, i.e.\ $c<0$, the function $s$ is defined on $(T_-,T_+)=(\frac{1}{2c},\infty)$, $s(0)=0$ and it is strictly increasing.  On the other hand, for shrinking solitons, i.e.\ $c>0$, $s$ is defined on $(T_-,T_+)=(-\infty,\frac{1}{2c})$, $s(0)=0$ and it is strictly decreasing.  Since $A$ is skew-symmetric, its eigenvalues are either purely imaginary numbers or zero, say $\pm\im a_1,\dots,\pm\im a_m,0,\dots,0$ ($a_j>0$).  If the set $\{ a_1,\dots,a_m\}$ is linearly dependent over $\QQ$, then there exists a sequence $t_k$, with $t_k \rightarrow \pm\infty$ (depending on the sign of $c$), such that $e^{s(t_k)A}=I$ for all $k$, and thus the bracket flow solution projected on the sphere,
$$
\frac{\mu(t)}{|\mu(t)|} =  \left[\begin{smallmatrix} I&0\\ 0& e^{s(t)A} \end{smallmatrix}\right] \cdot \lb,
$$
is periodic (see Example \ref{semi-ex} below).  Thus for any expanding (shrinking) semi-algebraic soliton, $\mu(t)$ converges to zero (to infinity) by rounding in a cone as $t\to\infty$ ($t\to\tfrac{1}{2c}$).  If on the contrary, the set $\{a_1,\dots,a_m\}$ is linearly independent over $\QQ$, then $\frac{\mu(t)}{|\mu(t)|}$ is not periodic and by Kronecker's theorem, for each $t_0\in(T_-,T_+)$, there exists a sequence $t_k$, with $t_k \rightarrow \pm\infty$ (depending on the sign of $c$), such that $e^{s(t_k)A} \rightarrow e^{s(t_0)A}$.  This implies that
$$
\frac{\mu(t_k)}{|\mu(t_k)|} \underset{k \rightarrow \infty}\longrightarrow\frac{\mu(t_0)}{|\mu(t_0)|},
$$
which reveals the following chaotic behavior: the solution projected on the sphere is not periodic, but nevertheless each point of the solution is contained in the $\omega$-limit.  The existence of a semi-algebraic soliton of this kind is an open problem.

\subsection{Laplacian flow diagonal property}
In this section, we aim to characterize algebraic solitons among homogeneous Laplacian solitons in a geometric way.  The concept of Laplacian soliton is a geometric invariant, that is, invariant under equivalence of $G_2$-structures (or pull-back by diffeomorphisms).  However, the concept of semi-algebraic soliton is not, as it may depend on the presentation of the homogeneous $G_2$-structure $(M,\vp)$ as a homogeneous space $(G/K,\vp)$.  Moreover, being an algebraic soliton may a priori not only depend on such presentation, but also on the reductive decomposition $\ggo = \kg \oplus \pg$ one is choosing for the homogeneous space.

\begin{definition}\label{LFdiag}
A homogeneous $G_2$-structure $(M,\vp)$ is said to be {\it Laplacian flow diagonal} if the $\Aut(M,\vp)$-invariant Laplacian flow solution $\vp(t)$ starting at $\vp$ satisfies the following property: at some point $p\in M$, there exists a basis $\beta$ of $T_pM$, orthonormal with respect to $\ip_\vp$, such that the matrix $[Q_{\vp(t)}(p)]_\beta$ is diagonal for all $t$.
\end{definition}

Recall from \eqref{defQphi} that $Q_{\vp(t)}\in\qg_{\vp(t)}\subset\End(TM)$ is the operator satisfying
$\theta(Q_{\vp(t)})\vp(t)=\Delta_{\vp(t)}\vp(t)$.  Therefore, if $(M,\vp)$ is Laplacian flow diagonal, then the primitive $3$-forms $e^i\wedge e^j\wedge e^k$ that appear in the formulas for the $3$-forms $\vp(t)$ and $\Delta_{\vp(t)}\vp(t)$ in terms of the basis $\beta=\{ e_1,\dots,e_7\}$ are the same for all time $t$.  Such a property is very convenient in the study of any aspect of the Laplacian flow ODE \eqref{LF2}, including its qualitative behavior and the search for exact solutions.  This dichotomy neatly arose in the study of Laplacian flow solutions on nilpotent Lie groups worked out in \cite[Section 4]{FrnFinMnr}: cases $N_2$ and $N_{12}$ are Laplacian flow diagonal, whereas $N_4$ and $N_6$ are not.

The following observations on the above definition are also in order:

\begin{itemize}
\item The point $p$ can be replaced by any other point by homogeneity.

\item The property of being Laplacian flow diagonal is invariant under equivalence since given $f\in\Diff(M)$, the operators corresponding to $(M,f^*\vp)$ are simultaneously conjugate via $df|_{f^{-1}(p)}$ to those of $(M,\vp)$.

\item For $\vp$ closed, the simultaneous diagonalization in the above definition is equivalent to the family of operators $\{ Q_{\vp(t)}:t\in (T_-,T_+)\}$ being commutative, as they are all symmetric by Proposition \ref{Qphi-form}.
\end{itemize}

Once we consider a presentation $(M,\vp)=(G/K,\vp)$ and a reductive decomposition $\ggo=\kg\oplus\pg$, Theorem \ref{BF-thm} tells us that $\vp(t)=h(t)^*\vp$ for the solution $h(t)\in\Gl(\pg)$ to the ODE $\ddt h(t)=-h(t)Q_{\vp(t)}$, $h(0)=I$.  It follows that the following conditions are equivalent:
\begin{itemize}
\item $(G/K,\vp)$ is Laplacian flow diagonal.

\item The family of operators $\{ h(t):t\in (T_-,T_+)\}$ is simultaneously diagonalizable with respect to an orthonormal basis of $\pg$.
\end{itemize}

\begin{example}
It follows from Theorem \ref{rsequiv} that if $(G/K,\vp)$ is a simply connected algebraic soliton, then $[Q_\vp,D_\pg]=0$ and so by \eqref{solgama},
$$
Q_{\vp(t)} = (-2ct+1)^{-1}e^{s(t)D_\pg} Q_{\vp} e^{-s(t)D_\pg}=(-2ct+1)^{-1}Q_\vp, \qquad\forall t.
$$
This shows that $\vp$ is Laplacian flow diagonal as soon as $Q_\vp$ is symmetric (or equivalently, diagonalizable with respect to an orthonormal basis), e.g.\ if $\vp$ is closed.
\end{example}

We now prove that the Laplacian flow diagonal condition actually characterizes algebraic solitons among homogeneous Laplacian solitons.

\begin{theorem}\label{diagalg}
A closed semi-algebraic soliton $(G/K,\vp)$ with $K$ compact is Laplacian flow diagonal if and only if it is an algebraic soliton.
\end{theorem}

\begin{remark}
In particular, any closed semi-algebraic soliton which is equivalent to an algebraic soliton must be an algebraic soliton itself.
\end{remark}

\begin{proof}
We can assume that condition \eqref{sas-D} holds for $(G/K,\vp)$.  By using that $Q_{\vp(t)} = (-2ct+1)^{-1}e^{s(t)D_\pg} Q_{\vp} e^{-s(t)D_\pg}$, we obtain from the Laplacian flow diagonal condition that
\[
[e^{sD_\pg}Q_{\vp}e^{-sD_{\pg}}, Q_{\vp}] = 0,  \qquad \forall s\in(-\epsilon,\epsilon).
\]
Now this implies that $[[D_\pg,Q_{\vp}],Q_{\vp}] = 0$, and so
$$
0=\tr{D_\pg[[D_\pg,Q_{\vp}],Q_{\vp}]}= -\tr{[D_\pg,Q_{\vp}]^2}.
$$
It follows that $[D_\pg,Q_{\vp}]=\unm[D_\pg,D_\pg^t]=0$ as it is symmetric, and thus $D_\pg$ is normal. Hence $D$ is normal as well relative to any extension of the inner product $\ip_\vp$ to $\ggo$, from which we obtain that $D^t\in \Der(\ggo)$ since it is well known that the transpose of a normal derivation of a metric Lie algebra is again a derivation. Thus $Q_\vp = cI + \unm(D+D^t)_\pg$, with $\unm(D+D^t) \in \Der(\ggo)$,  showing that $(G/K,\vp)$ is an algebraic soliton and concluding the proof.
\end{proof}

\section{Almost abelian solvmanifolds}\label{muA-sec}

We study in this section the Laplacian flow and its solitons in a class of solvable Lie groups which is relatively simple from the algebraic point of view but yet geometrically rich and exotic.

Let $(G,\vp)$ be a Lie group endowed with a left-invariant $G_2$-structure $\vp$.  Assume that the Lie algebra $\ggo$ of $G$ has a codimension-one abelian ideal $\hg$.  These Lie algebras are often called {\it almost abelian} in the literature.  It was proved in \cite{Frb1,Frb2} that there exists an orthonormal basis $\{ e_1,\dots,e_7\}$ of $\ggo$ with respect to $\ip_\vp$ such that $\hg=\spann\{e_1,\dots,e_6\}$ and
\begin{equation}\label{phiA}
\vp=\omega\wedge e^7+\rho^+=e^{127}+e^{347}+e^{567}+e^{135}-e^{146}-e^{236}-e^{245},
\end{equation}
where
$$
\omega:=e^{12}+e^{34}+e^{56}, \qquad \rho^+:=e^{135}-e^{146}-e^{236}-e^{245}.
$$
We note that each of these Lie algebras is completely determined by the real $6\times 6$ matrix
$$
A:=\ad{e_7}|_{\hg},
$$
and so its Lie bracket will be denoted by $\mu_A$ and the corresponding simply connected Lie group by $G_A$.

\begin{note}
Identification of linear maps with matrices will always be done via the ordered basis $\{ e_1,e_3,e_5,e_2,e_4,e_6,e_7\}$, unless otherwise stated.
\end{note}

Thus $\mu_A$ is solvable, $G_A$ is diffeomorphic to $\RR^7$, $\hg$ is always an abelian ideal (which is the nilradical of $\mu_A$ if and only if $A$ is not nilpotent) and $\mu_A$ is nilpotent if and only if $A$ is a nilpotent matrix.  It is not hard to see that $\mu_A$ is isomorphic to $\mu_B$ if and only if the matrices $A$ and $B$ are conjugate up to a nonzero scaling.

The nondegenerate $2$-form $\omega\in\Lambda^2\hg^*$ can be written as $\omega=\la J\cdot,\cdot\ra_\vp$, where
$$
J=\left[\begin{array}{c|c}
0&-I\\\hline
I& 0
\end{array}\right].
$$
Given $A\in\glg_6(\RR)$, let $(G_A,\vp)$ denote the Lie group $G_A$ endowed with the left-invariant $G_2$-structure defined by the fixed positive $3$-form $\vp$ given in \eqref{phiA}.

\begin{proposition}\label{formA}
$\mbox{ }$
\begin{itemize}
\item[(i)] \cite{Frb1} $(G_A,\vp)$ is closed if and only if the matrix $A$ belongs to
\begin{align*}
\slg(3,\CC):=&\left\{ A\in\glg_6(\RR):AJ=JA, \quad\tr{A}=\tr{AJ}=0\right\} \\
=&\left\{\left[\begin{array}{c|c}
B&-C\\\hline
C&B
\end{array}\right] : B,C\in\slg_3(\RR)\right\}.
\end{align*}

\item[(ii)] \cite{Frb2} $(G_A,\vp)$ is coclosed if and only if $A$ is in
\begin{align*}
\spg(3,\RR):=&\left\{ A\in\glg_6(\RR):A^tJ+JA=0\right\} \\
=& \left\{\left[\begin{array}{c|c}
B&C\\\hline
D&-B^t
\end{array}\right] : C,D\in\sym(3)\right\}.
\end{align*}
\end{itemize}
\end{proposition}

\begin{remark}
Part (ii) is proved in \cite{Frb2} for the $3$-form $\omega\wedge e^7-\rho^-$, but it is easily seen to be valid for $\vp$ as well (see Lemma \ref{tech1}, (iv) for the definition of $\rho^-$).
\end{remark}

We note that $\SU(3):=\left\{ h\in\SO(6):hJ=Jh\right\}=\Sl(3,\CC)\cap\Spe(3,\RR)$ can be homomorphically imbedded in $G_2=G_2(\vp)$ as
$$
\left[\begin{array}{c|c}
\SU(3) & 0 \\ \hline
0 & 1
\end{array}\right]\subset G_2.
$$
According to Proposition \ref{formA}, $(G_A,\vp)$ is torsion-free if and only if $A\in\sug(3)=\slg(3,\CC)\cap\spg(3,\RR)$.

Recall from Section \ref{homog-sec} that $(G_A,\vp)$ and $(G_B,\vp)$ are said to be equivariantly equivalent if they are equivalent as $G_2$-structures via a Lie group isomorphism, that is, if and only if there exists a Lie algebra isomorphism $h:(\ggo,\mu_A)\longrightarrow(\ggo,\mu_B)$ such that $h^*\vp=\vp$ (i.e.\ $h\in G_2$).

\begin{proposition}\label{iso-equiv}
If either $B=hAh^{-1}$ for some $h\in\SU(3)\subset G_2$, or $B=-hAh^{-1}$ for some $h\in\Or(6)$ such that $\det{h}=-1$ and $hJh^{-1}=-J$, then $(G_A,\vp)$ and $(G_B,\vp)$ are equivariantly equivalent.  The converse holds if neither $A$ nor $B$ are nilpotent.
\end{proposition}

\begin{remark}\label{iso-equiv-rem}
In the closed case, i.e.\ $A,B\in\slg(3,\CC)$, if we view all these matrices as complex $3\times 3$ matrices, then what the proposition is asserting is that $(G_A,\vp)$ and $(G_B,\vp)$ are equivariantly equivalent as soon as $B$ is $\SU(3)$-conjugate to $A$ or $\overline{A}$.
\end{remark}

\begin{proof}
Note that $(G_A,\vp)$ and $(G_B,\vp)$ are equivariantly equivalent if and only if $\mu_B=h\cdot\mu_A$ for some $h\in G_2\subset\SO(7)$.  We first prove the converse assertion.  If they are equivariantly equivalent, then it is easy to see by using that $h$ must leave $\hg$ invariant (notice that $\hg$ is the nilradical of both Lie algebras since $A$ and $B$ are both non-nilpotent) that such an $h$ must have the form
$$
h=\left[\begin{array}{c|c}
&\\ \quad h_1 \quad & 0 \\ &\\ \hline
0 & h_0
\end{array}\right], \qquad \mbox{for some} \quad h_1\in\Or(6), \quad h_0=\det{h_1}=\pm 1.
$$
Now condition $h^*\vp=\vp$ implies that $h_1^*\omega=h_0\omega$, which is equivalent to $h_1Jh_1^{-1}=h_0J$.  Since condition $\mu_B=h\cdot\mu_A$ is equivalent to $B=h_0h_1Ah_1^{-1}$, the two alternatives in the proposition correspond to $h_0=1$ and $h_0=-1$, respectively.

It is now clear that conversely, if we construct $h$ as above by setting $h_1$ as the map $h$ in the proposition, then $h$ defines an equivariant equivalence between $(G_A,\vp)$ and $(G_B,\vp)$, without any assumption on $A,B$, concluding the proof.
\end{proof}

The following result shows that two left-invariant $G_2$-structures on non-isomorphic Lie groups can indeed be equivalent (in particular, without being equivariantly equivalent).  Our proof is strongly based on the proof by Heber of \cite[Proposition 2.5]{Hbr} in his Habilitationsschrift (see \cite[Proposition 2.5]{Hbr2}).

\begin{proposition}\label{equivA}
Assume that $A=A_1+A_2$, $A_2\in\sug(3)$ and $[A_1,A_2]=0$.  Then the $G_2$-structures $(G_A,\vp)$ and $(G_{A_1},\vp)$ are equivalent.
\end{proposition}

\begin{proof}
We denote by $\ggo_A:=(\ggo,\mu_A)$, the Lie algebra of $G_A$.  Consider the Lie group
$$
F:=\Aut(G_A)\cap\Aut(G_A,\vp)\simeq\Aut(\ggo_A)\cap G_2,
$$
with Lie algebra $\fg:=\Der(\ggo_A)\cap\ggo_2$, the homomorphism $\alpha:\ggo_A\longrightarrow\fg$ defined by
$$
\alpha(e_7)=
\left[\begin{array}{c|c}
&\\ \quad -A_2 \quad & 0 \\ &\\ \hline
0 & 0
\end{array}\right], \qquad \alpha|_\hg\equiv 0,
$$
and denote also by $\alpha$ the corresponding Lie group homomorphism $G_A\longrightarrow F$.  By using that $G_A=\exp{\RR e_7}\ltimes\exp{\hg}$, it is easy to see that
$$
G_1:=\{ L_s\circ\alpha(s):s\in G_A\}\subset L(G_A)F\subset\Aut(G_A,\vp),
$$
is a subgroup, where $L:G_A\longrightarrow\Aut(G_A,\vp)$ is the left-multiplication morphism.  Thus $G_1$ is a connected and closed Lie subgroup of $\Aut(G_A,\vp)$ since $s\mapsto L_s\circ\alpha(s)$ is continuous and proper.  But $G_1$ acts simply and transitively on $G_A$ by automorphisms of $\vp$, so as usual, the diffeomorphism $f:G_1\longrightarrow G_A$, $L_s\circ\alpha(s)\mapsto (L_s\circ\alpha(s))(e)=s$, defines an equivalence between the left-invariant $G_2$-structures $(G_1,f^*\vp)$ and $(G_A,\vp)$.  On the other hand, the Lie algebra of $G_1$ is given by
$$
\ggo_1:=\left\{ dL|_eX+\alpha(X):X\in\ggo_A\right\}\subset\ggo_A\rtimes\fg,
$$
and if $X=X_\hg+ae_7$, $Y=Y_\hg+be_7$ belong to $\ggo_A$, then
\begin{align*}
[dL|_eX+\alpha(X),dL|_eY+\alpha(Y)] =& \mu_A(X,Y) + \alpha(X)Y - \alpha(Y)X + \alpha([X,Y]) \\
=& aA_1Y_\hg+aA_2Y_\hg-bA_1X_\hg-bA_2X_\hg \\
& - aA_2Y_\hg + bA_2X_\hg + 0 \\
=& \mu_{A_1}(X,Y) = (dL|_e+\alpha)\mu_{A_1}(X,Y).
\end{align*}

This shows that $\ggo_1$ is isomorphic to $\ggo_{A_1}$ and that $(G_1,f^*\vp)$ is equivariantly equivalent to $(G_{A_1},\vp)$, which implies that $(G_{A},\vp)$ and $(G_{A_1},\vp)$ are equivalent.
\end{proof}

\begin{remark}
If we replace $\vp$ by an inner product $\ip$ on $\ggo$, $\Aut(G_A,\vp)$ by $\Iso(G_A,\ip)$ and $G_2$ by $\Or(\ggo,\ip)$, then the following result can be proved in exactly the same way as above for any dimension: $(G_A,\ip)$ is isometric to $(G_{A_1},\ip)$ for any $A=A_1+A_2$ in $\glg(\hg)$ such that $[A_1,A_2]=0$ and $A_2\in\sog(\hg,\ip)$.
\end{remark}

If $\tr{A}=0$, then the Ricci operator and scalar curvature of $(G_A,\ip_\vp)$ are respectively given by (see e.g.\ \cite[(8)]{Arr})
\begin{equation}\label{RicA}
\Ricci_A  =  \left[\begin{array}{c|c}
& \\
\unm[A, A^t] &  0 \\ & \\\hline
0 & -\unc\tr{(A+A^t)^2}
\end{array}\right], \qquad R_A= -\unc\tr{(A+A^t)^2}.
\end{equation}
The following conditions are therefore equivalent for a closed (or coclosed) $(G_A,\vp)$:
\begin{itemize}
\item $(G_A,\vp)$ is torsion-free.

\item $A^t=-A$ (i.e.\ $A\in\sog(6)$).

\item $R_A=0$.

\item $\Ricci_A=0$.

\item $(G_A,\ip_\vp)$ is flat.

\item $(G_A,\vp)$ is equivalent (but not equivariantly equivalent unless $A=0$) to the $G_2$ euclidean space $(\RR^7,\vp)$ (see Proposition \ref{equivA}).
\end{itemize}

\begin{example}\label{n2n6}
The nonabelian nilpotent Lie groups $G_A$ admitting a closed $G_2$-structure (i.e. with $A$ conjugated to an element in  $\slg(3,\CC)$) are exactly two and their Lie algebras have been denoted by $\ngo_2$ ($A^2=0$) and $\ngo_6$ ($A^3=0$ and $A^2\ne 0$) in \cite{FrnFinMnr,Ncl} (see \cite{Frb2}).  Since any nonzero $A\in\slg(3,\CC)$ such that $A^2=0$ is $\SU(3)$-conjugate up to scaling to
$$
A_0:=\left[\begin{smallmatrix}
0&0&1\\
&0&0\\
&&0
\end{smallmatrix}\right],
$$
it follows from Proposition \ref{iso-equiv} and Remark \ref{iso-equiv-rem} that the Lie group $G_{A_0}$ with Lie algebra $\ngo_2$ admits only one closed $G_2$-structure up to equivalence and scaling.  On the other hand, the matrices
$$
A_t:=\left[\begin{smallmatrix}
0&t&0\\
&0&1\\
&&0
\end{smallmatrix}\right]\in\slg(3,\CC), \qquad t>0,
$$
provides a continuous family of $G_2$-structures $(G_{A_t},\vp)$, or equivalently a family $(G_{A_1},\vp_t)$, where the Lie algebra of $G_{A_1}$ is $\ngo_6$, such that there is no any pair which is equivalent up to scaling.  Indeed, by \eqref{RicA}, the Ricci operator of $(G_{A_t},\ip_\vp)$ is given by
$$
\Ricci_t=\unm\Diag(t^2,1-t^2,-1,t^2,1-t^2,-1,-2(1+t^2)),
$$
so the ratio between its extreme eigenvalues equals $-t^2/2(1+t^2)$, an injective function on $(0,\infty)$, which implies that two of these Riemannian manifolds can never be isometric up to scaling.
\end{example}

\begin{example}
Given a diagonal matrix with three different real eigenvalues,
$$
A:=\left[\begin{smallmatrix}
a&&\\
&b&\\
&&c
\end{smallmatrix}\right]\in\slg(3,\CC),
$$
we know that the set of all closed $G_2$-structures on $G_A$ is parameterized by its $\Gl(3,\CC)$-conjugacy class, which has dimension $12=18-6$.  By Proposition \ref{iso-equiv}, the equivalence classes are given by $\SU(3)$-conjugacy classes, so up to equivalence and scaling, the set of all closed $G_2$-structures on $G_A$ depends on $5=12-8+2-1$ parameters.
\end{example}

\begin{example}
In the coclosed case, for a diagonal matrix $A\in\spg(3,\RR)$ with six different real eigenvalues, we obtain from Proposition \ref{iso-equiv} that up to equivalence and scaling, the set of all coclosed $G_2$-structures on $G_A$ depends on $9=21-3-8+0-1$ parameters.
\end{example}

The following technical lemma contains some basic though very useful information on the linear algebra involved in subsequent computations.  Recall from \eqref{phiA} the definition of $\omega$ and $\rho^+$.

\begin{lemma}\label{tech1}
Let $\ast:\Lambda^k\ggo^*\longrightarrow\Lambda^{7-k}\ggo^*$ and $\ast_\hg:\Lambda^k\hg^*\longrightarrow\Lambda^{6-k}\hg^*$ be the Hodge star operators determined by $\vp$, i.e. by the ordered bases $\{ e_1,\dots,e_7\}$ and $\{ e_1,\dots,e_6\}$, respectively.

\begin{itemize}
\item[(i)] $\ast\gamma=\ast_\hg\gamma\wedge e^7$, for any $\gamma\in\Lambda^k\hg^*$.
\item[ ]
\item[(ii)] $\ast(\gamma\wedge e^7)=(-1)^k\ast_\hg\gamma$, for any $\gamma\in\Lambda^k\hg^*$.
\item[ ]
\item[(iii)] $\ast_\hg\omega=\unm\omega\wedge\omega$ and $\ast_\hg(\omega\wedge\omega)=2\omega$.
\item[ ]
\item[(iv)] $\ast_\hg\rho^+=\rho^-$ and $\ast_\hg\rho^-=-\rho^+$, where $\rho^-:= -e^{246}+e^{235}+e^{145}+e^{136}$.
\item[ ]
\item[(v)] $\ast^2=id$ and $\ast_\hg^2=(-1)^k id$ on $\Lambda^k\hg^*$.
\end{itemize}
\end{lemma}

Before computing the Hodge Laplacian operator, we give in the following lemma some properties of the differential of forms on the Lie group $G_A$.  Denote by $\theta:\glg(\hg)\longrightarrow\End(\Lambda^k\hg^*)$ the representation obtained as the derivative of the natural $\Gl(\hg)$-action on each $\Lambda^k\hg^*$, that is,
$$
\theta(A)\gamma = -\gamma(A\cdot,\dots,\cdot) - \dots -\gamma(\cdot,\dots,A\cdot), \qquad\forall\gamma\in\Lambda^k\hg^*.
$$

\begin{lemma}\label{tech2}
Let $d_A$ denote the differential of left-invariant forms on the Lie group $G_A$.

\begin{itemize}
\item [(i)] $d_Ae^7=0$ and
$$
d_Ae^i=\sum_{j=1}^6 a_{ij}e^{j7}, \quad i=1,\dots,6,
$$
where $A=[a_{ij}]$ is written in terms of the basis $\{ e_1,\dots,e_7\}$.
\item[ ]
\item[(ii)] $d_A\gamma=(-1)^k\theta(A)\gamma\wedge e^7$, for any $\gamma\in\Lambda^k\hg^*$.
\item[ ]
\item[(iii)] $d_A(\gamma\wedge e^7)=0$, for all $\gamma\in\Lambda^k\hg^*$.
\item[ ]
\item [(iv)] $\theta(A)\rho^+=0$ if and only if $\theta(A)\rho^-=0$, if and only if $A\in\slg(3,\CC)$.
\item[ ]
\item[(v)] $\theta(A)\omega=0$ if and only if $A\in\spg(3,\RR)$.
\item[ ]
\item[(vi)] $\theta(A)\ast_\hg = -\ast_\hg\theta(A^t)$ on $\Lambda\hg^*$, if $\tr{A}=0$.
\end{itemize}
\end{lemma}

\begin{proof}
Part (iv) follows from the fact that the complex volume form of $\CC^6$,
$$
\alpha = (e^1+\im e^2) \wedge (e^3+\im e^4) \wedge (e^5+\im e^6),
$$
can be written as $\alpha=\rho^++\im\rho^-$.  To prove part (vi), we first recall that
$$
\alpha\wedge\ast_\hg\beta = \la\alpha,\beta\ra \nu, \qquad \nu:=e^1\wedge\dots\wedge e^6, \qquad\forall\alpha,\beta\in\Lambda^k\hg^*.
$$
Thus, if $\alpha\in\Lambda^p\hg^*$ and $\beta\in\Lambda^{6-p}\hg^*$, then
\begin{align*}
\la\alpha,\theta(A)\ast_\hg\beta\ra\nu =& \la\theta(A^t)\alpha,\ast_\hg\beta\ra\nu = \theta(A^t)\alpha\wedge\ast_\hg^2\beta \\
=& (-1)^p\theta(A^t)\alpha\wedge\beta = (-1)^{p+1}\alpha\wedge\theta(A^t)\beta  \\
=& -\alpha\wedge\ast_\hg\ast_\hg\theta(A^t)\beta = \la\alpha,-\ast_\hg\theta(A^t)\beta\ra\nu.
\end{align*}
We have used in the second line above that $\theta(A)\nu=0$ (recall that $\tr{A}=0$) and so $\theta(A^t)(\alpha\wedge\beta)=0$.  The other parts of the lemma easily follow.
\end{proof}

\begin{proposition}\label{DeltaA}
Let $\Delta_A$ denote the Hodge Laplacian operator of $(G_A,\vp)$.  If $\tr{A}=0$, then
$$
\Delta_A\vp = \theta(A)\theta(A^t)\omega\wedge e^7 - \theta(A^t)\theta(A)\rho^+.
$$
\end{proposition}

\begin{proof}
In the following computations, we are using many of the properties and identities given in Lemmas \ref{tech1} and \ref{tech2} without any further mention.  One has that,
\begin{align*}
\vp =& \omega\wedge e^7 + \rho^+, \\
d_A\vp =& d_A\rho^+ = -\theta(A)\rho^+\wedge e^7, \\
\ast d_A\vp =& \ast_\hg\theta(A)\rho^+ = -\theta(A^t)\ast_\hg\rho^+ = -\theta(A^t)\rho^-, \\
d_A\ast d_A\vp =& \theta(A)\theta(A^t)\rho^-\wedge e^7, \\
\ast d_A\ast d_A\vp =& -\ast_\hg\theta(A)\theta(A^t)\rho^- = -\theta(A^t)\theta(A)\ast_\hg\rho^- = \theta(A^t)\theta(A)\rho^+.
\end{align*}

On the other hand,
\begin{align*}
\ast\vp =& \ast_\hg\omega+\ast_\hg\rho^+\wedge e^7 = \unm\omega\wedge\omega + \rho^-\wedge e^7,  \\
d_A\ast\vp =& = \unm\theta(A)(\omega\wedge\omega)\wedge e^7 =\theta(A)\ast_\hg\omega\wedge e^7 = -\ast_\hg\theta(A^t)\omega\wedge e^7, \\
\ast d_A\ast\vp =& -\ast_\hg^2\theta(A^t)\omega = -\theta(A^t)\omega, \\
d_A\ast d_A\ast\vp =& -\theta(A)\theta(A^t)\omega\wedge e^7,
\end{align*}
which concludes the proof.
\end{proof}

\begin{remark}
Second lines in the two computations above provide a proof for Proposition \ref{formA}.
\end{remark}

\begin{proposition}
If $(G_A,\vp)$ is closed, then the symmetric operator $Q_A\in\sym(7)$ satisfying $\theta(Q_A)\vp=\Delta_A\vp$ is given by
\begin{equation}\label{QA}
Q_A = \left[\begin{array}{c|c}
& \\
\quad Q_1\quad  &  0 \\ & \\\hline
0 & q
\end{array}\right],
\end{equation}
where
$$
\quad Q_1=\unm[A,A^t]+\frac{1}{12}\tr{(A+A^t)^2}I-\unm(A+A^t)^2, \qquad q=-\frac{1}{6}\tr{(A+A^t)^2}.
$$
\end{proposition}

\begin{proof}
It follows from Proposition \ref{DeltaA} that
\begin{align*}
\Delta_A\vp =& \theta(A)\theta(A^t)\omega\wedge e^7 \\
=& \left(\omega(A^tA\cdot,\cdot)+\omega(A^t\cdot,A\cdot)+\omega(A\cdot,A^t\cdot)+\omega(\cdot,A^tA\cdot)\right)\wedge e^7.
\end{align*}
Now by using that $\omega(A\cdot,\cdot)=\omega(\cdot,A^t\cdot)$ (recall that $AJ=JA$), we obtain
\begin{align*}
\Delta_A\vp =& \omega\left((A^tA+(A^t)^2+A^2+A^tA)\cdot,\cdot\right)\wedge e^7 \\
=& \omega\left((-[A,A^t]+(A+A^t)^2)\cdot,\cdot\right)\wedge e^7 \\
=& \omega\left(B\cdot,\cdot\right)\wedge e^7 + \frac{1}{6}\tr{(A+A^t)^2}\omega\wedge e^7,
\end{align*}
where $B:=-[A,A^t]-\frac{1}{6}\tr{(A+A^t)^2}I+(A+A^t)^2\in\slg(3,\CC)$ (i.e. $\theta(B)\rho^+=0$) and $B^t=B$.  This implies that
$$
\Delta_A\vp = \theta\left(\left[\begin{array}{c|c}
& \\
-\unm B &  0 \\ & \\\hline
0 & -\frac{1}{6}\tr{(A+A^t)^2}
\end{array}\right]\right)(\omega\wedge e^7+\rho^+),
$$
as was to be shown.

\vspace{.2cm}
\noindent {\it Alternative proof}.  We have seen in the proof of Proposition \ref{DeltaA} that the torsion $2$-form $\tau_A=-\ast d_A\ast\vp\in\Lambda^2\ggo^*$ of $(G_A,\vp)$ is given by $\tau_A=\theta(A^t)\omega$, thus $\tau_A(e_7,\cdot)\equiv 0$ and
$$
\tau_A|_{\hg\times\hg}=\theta(A^t)\omega = -\la JA^t\cdot,\cdot\ra_\vp -\la J\cdot,A^t\cdot\ra_\vp = -\la J(A+A^t)\cdot,\cdot\ra_\vp.
$$
Therefore, as matrices, $\tau_A\in\sog(7)\equiv\Lambda^2\ggo^*$ and its square are respectively given by
\begin{equation}\label{tauA}
\tau_A = \left[\begin{array}{c|c}
& \\
-J(A+A^t) &  0 \\ & \\\hline
0 & 0
\end{array}\right],  \qquad \tau_A^2 = \left[\begin{array}{c|c}
& \\
-(A+A^t)^2 &  0 \\ & \\\hline
0 & 0
\end{array}\right].
\end{equation}
Formula \eqref{QA} now follows from Proposition \ref{Qphi-form}, \eqref{RicA} and \eqref{tauA}, concluding the proof.
\end{proof}

\subsection{Bracket flow}\label{BFA-sec}
We study in this section the bracket flow evolution of closed $G_2$-structures $(G_A,\vp)$ (see Section \ref{BF-sec}).  Recall from Section \ref{hm} the variety $\lca$ of $7$-dimensional Lie algebras.

\begin{proposition}\label{invf1}
The family $\left\{\mu_A : A\in\slg(3,\CC)\right\}\subset\lca$ of closed $G_2$-structures is invariant under the bracket flow, which becomes equivalent to the following ODE for a one-parameter family of matrices $A=A(t)\in\slg(3,\CC)$:
\begin{equation}\label{BFA}
\ddt A = -\frac{1}{6}\tr{(A+A^t)^2}A + \unm[A,[A,A^t]] - \unm[A,(A+A^t)^2].
\end{equation}
\end{proposition}

\begin{proof}
We first note that the family
$$
\left\{\mu_A : A\in\glg_6(\RR)\right\}\subset\lca,
$$
is invariant under the bracket flow $\ddt\mu=\delta_\mu(Q_\mu)$ if and only if the velocity $\delta_{\mu_A}(Q_A)$ equals $\mu_B$ for some $B\in\glg_6(\RR)$, for any $A$.  Using \eqref{QA}, it is easy to see that this indeed holds for $B=qA+[A,Q_1]$.  Note that if $A$ is in $\slg(3,\CC)$, then $B$ is so, since $Q_1$ is symmetric and traceless and thus $Q_1\in\slg(3,\CC)$.  Thus the subset of closed $G_2$-structures is invariant under the bracket flow, which takes the form $\ddt A=B$, as was to be shown.
\end{proof}

\begin{remark}
Equation \eqref{BFA} is substantially different from the bracket flow used by Arroyo (see \cite[(7)]{Arr}) to study the Ricci flow for Riemannian manifolds $(G_A,\ip_\vp)$, which is given in that case by $\ddt A = -\frac{1}{4}\tr{(A+A^t)^2}A + \unm[A,[A,A^t]]$ if $\tr{A}=0$.
\end{remark}

Since for each $t$ the Lie algebra $\mu_{A(t)}$ is isomorphic to the starting point $\mu_{A_0}$, we have that
$$
A(t)=a(t)h(t)A_0h(t)^{-1}, \qquad\mbox{for some} \quad a(t)>0, \quad h(t)\in\Sl(3,\CC),
$$
where $h$ and $c$ are smooth functions with  $a(0)=1$, $h(0)=I$.  The corresponding spectra (i.e. the unordered set of complex eigenvalues) therefore satisfy
\begin{equation}\label{spec}
\Spec(A(t))=a(t)\Spec(A_0), \qquad\forall t\in(T_-,T_+).
\end{equation}
This implies that if $c_kA(t_k)\to B$ for some subsequence $t_k\to T_\pm$ as in the hypothesis of Corollary \ref{i-conv}, then either the limit Lie group $G_B$ is isomorphic to $G_{A_0}$ or $B$ is nilpotent.

\begin{lemma}
Along the bracket flow \eqref{BFA}, the norm of $A(t)\in\slg(3,\CC)$ evolves by
$$
\ddt|A|^2=-\frac{1}{3}|A|^2|A+A^t|^2-|[A,A^t]|^2.
$$
\end{lemma}

\begin{proof}
From Proposition \ref{invf1} we obtain that
\begin{align*}
\ddt|A|^2 =& 2\left\la\ddt A,A\right\ra = 2\tr{\left(\ddt A\right)A^t} \\
=& -\frac{1}{3}|A|^2|A+A^t|^2 -|[A,A^t]|^2 + \tr{(A+A^t)^2[A,A^t]},
\end{align*}
and since
\begin{equation}\label{ortA}
\tr{(A+A^t)^2[A,A^t]} = \tr{(AA^t+A^tA)(AA^t-A^tA)} =0,
\end{equation}
the lemma follows.

For an alternative proof, one can use \cite[Proposition 3]{BF}, asserting that the norm of the bracket evolves under the bracket flow by $\ddt|\mu|^2=-8\tr{Q_\mu M_\mu}$, where $M_\mu$ is the moment map of $(G_A,\ip_\vp)$.  In this case, $|\mu_A|^2=2|A|^2$ and a straightforward computation gives
$$
M_A:=M_{\mu_A} = \left[\begin{array}{c|c}
& \\
\unm[A,A^t] &  0 \\ & \\\hline
0 & -\unm|A|^2
\end{array}\right],
$$
so the lemma follows from \ref{QA}.
\end{proof}

From the above lemma, one has that $|A(t)|^2$ is non-increasing and so long-time existence for the bracket flow follows.  Actually, $|A(t)|^2$ is strictly decreasing unless $(G_A,\vp)$ is torsion-free (i.e.\ $A^t=-A$), in which case $A(t)\equiv A_0$.  In view of the equivalence between the bracket flow and the Laplacian flow given in Theorem \ref{BF-thm}, we obtain that Laplacian flow solutions among this class are all immortal.

\begin{corollary}
The left-invariant Laplacian flow solution starting at any closed $G_2$-structure $(G_A,\vp)$ is defined for all $t\in(T_-,\infty)$ for some $T_-<0$.
\end{corollary}

Next example shows that $A(t)$ does not necessarily converge to zero.

\begin{example}
The family
$$
A=\left[\begin{array}{c|c}
B&0\\\hline
0&B\\
\end{array}\right]\in\slg(3,\CC), \qquad B=\left[\begin{smallmatrix}
0&a&0\\
b&0&0\\
0&0&0
\end{smallmatrix}\right], \qquad a,b\in\RR,
$$
is invariant under the bracket flow \eqref{BFA} and it is easily seen to evolve by
$$
\left\{\begin{array}{l}
a'=\frac{2}{3}a(-2a^2-ab+b^2), \\ \\
b'=\frac{2}{3}b(-2b^2-ab+a^2).
\end{array}\right.
$$
A standard qualitative analysis gives the following behaviors for the solutions:

\begin{itemize}
\item Any point in the line $b=-a$ is a torsion-free $G_2$-structure which is therefore a fixed point for the flow.

\item The lines $a=0$, $b=0$ and $b=a$ are all invariant and each one contains two trajectories converging to zero.  These are precisely the algebraic solitons among the family (see Theorem \ref{rsequiv}).

\item The solutions in the second and fourth quadrant all converge to some fixed point in the line $b=-a$.

\item Any solution in the first or third quadrant converges to zero and asymptotically approaches the soliton line $b=a$.
\end{itemize}
\end{example}

\begin{proposition}\label{Revol}
The scalar curvature $R(t)=R(g_{\vp(t)})$ of the left-invariant Laplacian flow solution $\vp(t)$ starting at a nonflat closed $G_2$-structure $(G_{A_0},\vp)$ strictly increases and satisfies that
$$
\frac{1}{-2t+\frac{1}{R(0)}}\leq R(t) < 0, \qquad \forall t\in (T_-,\infty).
$$
In particular, $|\tau_{\vp(t)}|^2=-2R(t)$ is strictly decreasing and converges to zero, as $t\to\infty$.
\end{proposition}

\begin{proof}
It follows from Proposition \ref{invf1} that
\begin{align*}
\ddt\tr{(A+A^t)^2} =& 2\tr{(A+A^t)\ddt (A+A^t)} \\
=& -\frac{1}{3}\left(\tr{(A+A^t)^2}\right)^2 + \tr{(A+A^t)[A-A^t,[A,A^t]]} \\
& - \tr{(A+A^t)[A-A^t,(A+A^t)^2]} \\
=& -\frac{1}{3}\left(\tr{(A+A^t)^2}\right)^2 - \tr{[A-A^t,A+A^t][A,A^t]} \\
=& -\frac{1}{3}\left(\tr{(A+A^t)^2}\right)^2 -2 \tr{[A,A^t]^2} \\
\leq& -\frac{1}{3}\left(\tr{(A+A^t)^2}\right)^2.
\end{align*}
This implies that
$$
\tr{(A+A^t)^2}\leq \frac{1}{\unm t+\frac{1}{\tr{(A_0+A_0^t)^2}}}, \qquad \forall t\in (T_-,\infty),
$$
and so the proposition follows from the formula $R_A=-\unc\tr{(A+A^t)^2}$ in \eqref{RicA}.
\end{proof}

\subsection{Solitons}
We now give necessary and sufficient conditions on the matrix $A\in\slg(3,\CC)$ to obtain an algebraic soliton.

\begin{proposition}\label{sol-A-prop}
A closed $G_2$-structure $(G_A,\vp)$ is an algebraic soliton if and only if either $A$ is normal (i.e.\ $[A,A^t]=0$) or $A$ is nilpotent and
\begin{equation}\label{sol-A}
[A,[A,A^t]-(A+A^t)^2] = -\frac{|[A,A^t]|^2}{|A|^2}A.
\end{equation}
In any case, one has that $D=Q_A-cI\in\Der(\mu_A)$ and $\Delta_A\vp=-3c\vp - \lca_{X_D}\vp$ for
$$
c=-\frac{1}{6}\tr{(A+A^t)^2}-\frac{|[A,A^t]|^2}{2|A|^2}.
$$
In particular, they are all expanding Laplacian solitons, unless they are torsion-free (i.e.\ $A^t=-A$).
\end{proposition}

\begin{proof}
Assume that $(G_A,\vp)$ is an algebraic soliton, i.e.\ $Q_A=cI+D$ for some $c\in\RR$ and $D\in\Der(\ggo,\mu_A)$.  Thus $De_7=de_7$ for some $d\in\RR$ and  $[Q_1,A]=[D|_\hg,A]=dA$.  It follows from \eqref{QA} that
$$
\unm[A,[A,A^t] - \unm[A,(A+A^t)^2] = -dA,
$$
and thus $-d\tr{A^k}=0$ for any $k\in\NN$.  On the other hand, from \eqref{ortA} we obtain that $d=|[A,A^t]|^2/(2|A|^2)$, so either $A$ is normal or $A$ is nilpotent and satisfies the matrix equation in the proposition.  The converse can be easily checked and the formula for $c$ follows from the fact that $c=q-d$.
\end{proof}

\begin{corollary}\label{sol-A-cor}
Any closed algebraic soliton of the form $(G_A,\vp)$ is equivariantly equivalent up to scaling to one of the following:
\begin{itemize}
\item[(i)] $A=\left[\begin{smallmatrix}
x&&\\
&y&\\
&&z
\end{smallmatrix}\right]$, $x,y,z\in\CC$, $x+y+z=0$.

\item[(ii)] $A=\left[\begin{smallmatrix}
0&0&1\\
&0&0\\
&&0
\end{smallmatrix}\right]$.
\end{itemize}
Two $G_2$-structures in part (i) are equivariantly equivalent if and only if either $\{ x_2,y_2,z_2\}=\{ x_1,y_1,z_1\}$ or $\{ x_2,y_2,z_2\} = \{ \overline{x_1},\overline{y_1},\overline{z_1}\}$.
\end{corollary}

\begin{remark}
In relation to the question studied in \cite{FrnFinMnr}, we observe that the metrics attached to all the above closed Laplacian solitons are well known to be Ricci solitons (see e.g.\ \cite{Arr}).
\end{remark}

\begin{proof}
Let $(G_A,\vp)$ be an algebraic soliton which is closed, so $A\in\slg(3,\CC)$.  If $A$ is normal, then on its $\SU(3)$-conjugation class there is a diagonal matrix, so $(G_A,\vp)$ is equivalent to one of the structures in part (i) by Proposition \ref{iso-equiv} and Remark \ref{iso-equiv-rem}.  The last statement on equivalence among these structures follows in much the same way.

Assume now that $A$ is nilpotent and nonzero.  It follows that $A$ is $\SU(3)$-conjugate to a matrix of the form,
$$
\left[\begin{smallmatrix}
0&a&b\\
0&0&c\\
0&0&0
\end{smallmatrix}\right], \qquad a,b,c\in\CC.
$$
It is straightforward to check that condition \eqref{sol-A} holds for such a matrix if and only if $a=0$ or $c=0$ (recall that the transpose $A^t$ must be replaced by $A^*=\overline{A^t}$ in condition \eqref{sol-A} when working with complex matrices).  The resulting matrix is easily seen to be $\SU(3)$-conjugate up to scaling to the one given in part (ii), concluding the proof.
\end{proof}

We note that the closed Laplacian solitons provided by part (i) in the above corollary are given as real matrices by
$$
A=\left[\begin{array}{c|c}
B&-C\\\hline
C&B\\
\end{array}\right], \quad B=\left[\begin{smallmatrix}
a&&\\
&b&\\
&&c
\end{smallmatrix}\right], \; C=\left[\begin{smallmatrix}
d&&\\
&e&\\
&&f
\end{smallmatrix}\right]\in\slg_3(\RR).
$$

\begin{corollary}
Let $G_A$ be an almost abelian Lie group such that the matrix $A\in\glg_6(\RR)$ is conjugate to an element in $\slg(3,\CC)$.  Then $G_A$ admits a left-invariant closed algebraic soliton if and only if either $A$ is semisimple or $A^2=0$.  Such a soliton is the unique algebraic soliton up to equivariant equivalence and scaling among all left-invariant closed $G_2$-structures on $G_A$.
\end{corollary}

\begin{proof}
The existence assertion follows from the above corollary.  For the uniqueness, we use that on each $\Sl(3,\CC)$-conjugation class of semisimple matrices in $\slg(3,\CC)$ there is a unique $\SU(3)$-conjugacy class of normal matrices and Proposition \ref{iso-equiv}.
\end{proof}

\begin{remark}\label{n2n6-2}
Concerning the nilpotent case described in Example \ref{n2n6}, it follows from the above corollaries that $\ngo_2$ admits a unique closed algebraic soliton and that there is no any closed algebraic soliton on $\ngo_6$.
\end{remark}

In what follows, we explore the existence of semi-algebraic solitons of the form $(G_A,\vp)$.

\begin{proposition}\label{sas-A}
A closed $G_2$-structure $(G_A,\vp)$ is a semi-algebraic soliton if and only if either $A$ is normal and it is an algebraic soliton, or $A$ is nilpotent and
\begin{equation}\label{sa-sol-A}
[A,A^t]-(A+A^t)^2 = -\left(2d+\frac{1}{2}\tr{(A+A^t)^2}\right)I + D_1 + D_1^t, \qquad d:=\frac{|[A,A^t]|^2}{2|A|^2},
\end{equation}
for some $D_1\in\glg(\hg)$ such that $[D_1,A]=dA$.  This is an algebraic soliton if and only if also $[D_1^t,A]=dA$.  In any case, one has that $Q_A=cI+\unm(D+D^t)$ for $D\in\Der(\mu_A)$ given by $D|_\hg=D_1$, $De_7=de_7$, and $\Delta_A\vp=-3c\vp - \lca_{X_D}\vp$ for
$$
c=-\frac{1}{6}\tr{(A+A^t)^2}-\frac{|[A,A^t]|^2}{2|A|^2}.
$$
In particular, they are all expanding Laplacian solitons, unless they are torsion-free (i.e.\ $A^t=-A$).
\end{proposition}

\begin{proof}
Assume that $(G_A,\vp)$ is a semi-algebraic soliton, i.e.\ $Q_A=cI+\unm(D+D^t)$ for some $c\in\RR$ and $D\in\Der(\mu_A)$.  It follows from \eqref{QA} that $De_7=de_7$ for some $d\in\RR$, so $[D_1,A]=dA$ for $D_1:=D|_\hg$.  This is actually $D|_\hg$ composed with the projection on $\hg$ since it may happen that $D\hg$ is not contained in $\hg$.  Note that in any case,
$$
\left[\begin{array}{c|c}
& \\
\quad D_1\quad  &  0 \\ & \\\hline
0 & d
\end{array}\right]
$$
is also a derivation of $\mu_A$.  Furthermore, one has that $q=c+d$ and thus condition \eqref{sa-sol-A} holds since $q+2c=3q-2d$.  Finally, the formula for $d$ can be obtained as follows using \eqref{ortA}:
$$
d|A|^2 =\tr{A^t[D_1,A]} = \tr{[A,A^t]D_1} = \unm\tr{[A,A^t]^2}.
$$
In particular, if $A$ is normal then $d=0$ and $[D_1^t,A]=0$, so $D^t$ is also a derivation and $(G_A,\vp)$ is an algebraic soliton.  The converse can be easily checked and the formula for $c$ follows from the fact that $c=q-d$.
\end{proof}

Recall from Remark \ref{n2n6-2} that $\ngo_6$ (i.e. $A^3=0$ and $A^2\ne 0$) does not admit any algebraic soliton.  However, we now show that $\ngo_6$ does admit a semi-algebraic soliton, which was first found in \cite{Ncl} by different methods.

\begin{example}\label{semi-ex}
With the same notation as in the above proposition, consider the nilpotent matrix
$$
A=\left[\begin{array}{c|c}
B&0\\\hline
0&B\\
\end{array}\right]\in\slg(3,\CC), \qquad B=\left[\begin{smallmatrix}
0&1&0\\
0&0&\sqrt{2}\\
0&0&0
\end{smallmatrix}\right],
$$
and
$$
D_1=\left[\begin{array}{c|c}
D_2&0\\\hline
0&D_2\\
\end{array}\right], \qquad D_2=\left[\begin{smallmatrix}
4&0&-\sqrt{2}\\
0&3&0\\
0&0&2
\end{smallmatrix}\right], \qquad d=1.
$$
Thus $\mu_A$ is isomorphic to $\ngo_6$.  A straightforward computation gives that condition \eqref{sa-sol-A} holds and $[D_2,B]=B$, so the closed $G_2$-structure $(G_A,\vp)$ is an expanding semi-algebraic soliton with
$$
Q_A=-3I + \unm(D+D^t), \qquad D\in\Der(\mu_A).
$$
Since $[D_2^t,B]\ne B$, $D^t$ is not a derivation of $\mu_A$ and thus $(G_A,\vp)$ is not an algebraic soliton.  We note that $(G_A,\ip_\vp)$ is not a Ricci soliton (see e.g.\ \cite{Arr}) and the closed $G_2$-structure on $\ngo_6$ inducing a Ricci soliton obtained in \cite{FrnFinMnr} corresponds to
$$
A'=\left[\begin{array}{c|c}
B'&0\\\hline
0&B'\\
\end{array}\right]\in\slg(3,\CC), \qquad B'=\left[\begin{smallmatrix}
0&1&0\\
0&0&1\\
0&0&0
\end{smallmatrix}\right].
$$
From Proposition \ref{sas-der}, we know that the genuine Laplacian flow solution starting at $(G_A,\vp)$ is given by $\vp(t)=b(t)e^{s(t)D}\cdot\vp$, where $b(t)=(6t+1)^{3/2}$, so it has the form
$$
\vp(t) = e^{127}+e^{2s(t)}e^{347}+f(t)e^{567}+e^{135}-e^{146}-e^{236}-e^{245} + g(t)(e^{167}-e^{257}),
$$
for some smooth functions $f(t)$ and $g(t)$ (compare with the solution $\vp_6(t)$ in the proof of \cite[Theorem 4.8]{FrnFinMnr}).  The appearance of the last two new primitive $3$-forms $e^{167}$ and $e^{257}$ reveals the fact that $(G_A,\vp)$ is not Laplacian flow diagonal (see Theorem \ref{diagalg}).

According to \eqref{saBF2}, the bracket flow solution starting at $\mu_A$ is given by
$$
\mu(t)=(6t+1)^{-1/2} e^{s(t)E}\cdot\mu_A, \qquad s(t)=\tfrac{1}{6}\log(6t+1), \qquad t\in(-\tfrac{1}{6},\infty),
$$
where
$$
E:=\unm(D_2-D_2^t)=\left[\begin{array}{c|c|c}
E_2&0&\\\hline
0&E_2&\\\hline
&&0
\end{array}\right], \qquad E_2:=\tfrac{1}{\sqrt{2}}\left[\begin{smallmatrix}
0&0&-1\\
0&0&0\\
1&0&0
\end{smallmatrix}\right].
$$
This implies that the matrix bracket flow solution to \eqref{BFA} starting at $A$ equals
$$
A(t)=(6t+1)^{-1/2} e^{s(t)E}Ae^{-s(t)E} = (6t+1)^{-1/2}\left(\cos{\tfrac{s(t)}{\sqrt{2}}}A + \sin{\tfrac{s(t)}{\sqrt{2}}}A^{\perp}\right),
$$
where
$$
A^{\perp}:=\left[\begin{array}{c|c}
B^{\perp}&0\\\hline
0&B^{\perp}\\
\end{array}\right]\in\slg(3,\CC), \qquad B^{\perp}=\left[\begin{smallmatrix}
0&0&0\\
-\sqrt{2}&0&0\\
0&1&0
\end{smallmatrix}\right].
$$
We therefore obtain that $A(t)/|A(t)|$ runs on a circle and $A(t)$ converges to zero rounding in a cone.
\end{example}

\begin{example}
It is straightforward to see that the $4$-dimensional subspace of $\slg(3,\CC)$ given by
$$
A=\left[\begin{array}{c|c}
B&0\\\hline
0&B\\
\end{array}\right]\in\slg(3,\CC), \qquad B=\left[\begin{smallmatrix}
0&a&0\\
c&0&b\\
0&d&0
\end{smallmatrix}\right], \qquad a,b,c,d\in\RR,
$$
is invariant under the bracket flow \eqref{BFA}, and that the evolution is equivalent to the following dynamical system on $\RR^4$:
\begin{align*}
a' =& -\tfrac{5}{3}a^3-\tfrac{11}{6}abd-\tfrac{4}{3}a^2c+dcb+\tfrac{1}{3}ac^2-\tfrac{2}{3}ab^2-\tfrac{5}{3}ad^2+\tfrac{1}{2}cd^2, \\
b' =& -\tfrac{5}{3}b^3+\tfrac{1}{3}ba^2+\tfrac{1}{3}bd^2-\tfrac{5}{6}acb-\tfrac{5}{3}c^2b-\tfrac{1}{2}dc^2-\tfrac{4}{3}db^2, \\
c' =& -\tfrac{5}{3}c^3+\tfrac{1}{3}a^2c-\tfrac{5}{6}dcb-\tfrac{4}{3}ac^2-\tfrac{1}{2}ab^2-\tfrac{5}{3}cb^2+\tfrac{1}{3}cd^2, \\
d' =& -\tfrac{5}{3}d^3-\tfrac{11}{6}dca+\tfrac{1}{2}ba^2-\tfrac{4}{3}bd^2+acb-\tfrac{2}{3}dc^2+\tfrac{1}{3}db^2-\tfrac{5}{3}da^2.
\end{align*}
The hypersurface $ac+bd=0$ of nilpotent matrices is invariant under this ODE system, as well as the subset of $2$-step nilpotent ones (i.e.\ $a=d=0$ or $b=c=0$).  According to Corollary \ref{sol-A-cor}, such a subset is formed by all algebraic solitons isomorphic to $\ngo_2$ and so they evolve on a straight line converging to zero (see Theorem \ref{rsequiv}).  On the other hand, each point in the complement of $3$-step nilpotent matrices is isomorphic to $\ngo_6$.  We know from Example \ref{semi-ex} that the $2$-subspace generated by $(1,\sqrt{2},0,0)$ and $(0,0,-\sqrt{2},1)$ consists of semi-algebraic solitons on $\ngo_6$ which are not algebraic and that they flow by rounding on a cone while they converge to zero.  It would be really interesting to understand the qualitative behavior of the solutions outside the three $2$-dimensional subspaces above, specially in relation with possible chaotic behaviors.
\end{example}

We now exhibit an example of a homogeneous Laplacian soliton which is not semi-algebraic when presented as a left-invariant $G_2$-structure on a Lie group.

\begin{example}\label{no-sas}
Consider $A=A_1+A_2\in\slg(3,\RR)$, where
$$
A_1:=\left[\begin{smallmatrix}
0&0&1\\
0&0&0\\
0&0&0
\end{smallmatrix}\right]\in\slg(3,\RR), \qquad
A_2:= \left[\begin{smallmatrix}
\im&0&0\\
0&-2\im&0\\
0&0&\im
\end{smallmatrix}\right]\in\sug(3).
$$
According to Proposition \ref{sas-A}, $(G_A,\vp)$ is not a semi-algebraic soliton as $A$ is neither normal nor nilpotent.  Moreover, $G_A$ does not admit any left-invariant semi-algebraic soliton since $A$ is neither semisimple nor nilpotent.  However, it follows from Proposition \ref{equivA} that the closed $G_2$-structure $(G_A,\vp)$ is equivalent to the algebraic soliton $(G_{A_1},\vp)$ (see Corollary \ref{sol-A-cor}, (ii)), so $(G_A,\vp)$ is indeed a Laplacian soliton.
\end{example}

\subsection{Compact quotients}
It is well known that the Lie group $G_A$ admits a lattice (i.e. a cocompact discrete subgroup) if and only if
$$
\sigma e^{\alpha A}\sigma^{-1}\in\Sl_{6}(\ZZ),
$$
for some nonzero $\alpha\in\RR$ and $\sigma\in\Gl_{6}(\RR)$ (see e.g. \cite[Section 4]{Bck} or \cite{Hng}).  In that case, an example of a lattice is given by
$$
\Gamma=\exp\left(\sigma^{-1}\ZZ^{6}\rtimes\ZZ\alpha e_{7}\right).
$$
Thus $G_A$ always admits a lattice when $A$ is nilpotent and if $A$ is semisimple, then $G_A$ has a lattice if and only if the set $\{ e^{\alpha x},e^{\alpha y},e^{\alpha z}\}$, where $\{ x,y,z\}$ are the eigenvalues of $A$, is for some nonzero $\alpha\in\CC$ a set of conjugate algebraic units, i.e. the set of zeroes of a polynomial of the form $t^3+at^2+bt-1$, for some $a,b\in\ZZ$.

In particular, many of the closed Laplacian solitons of the form $(G_A,\vp)$ given in this section (see Corollary \ref{sol-A-cor} and Example \ref{semi-ex}) do admit compact quotients.  However, the corresponding closed $G_2$-structure on the compact manifold $M=G_A/\Gamma$ is not necessarily a Laplacian soliton since the vector field $X_D$ does not descend to $M$.  The Laplacian flow solution $\vp(t)$ on $M$ starting at $\vp$ remains locally equivalent to $\vp$, is immortal and has apparently no chances to converge in any reasonable sense, though the norm of the intrinsic torsion of $\vp(t)$ converges to zero, as $t\to\infty$ (see Proposition \ref{Revol}).

\end{document}